\documentclass[11pt,a4paper,leqno,noamsfonts]{amsart}
\usepackage[english]{babel}
\usepackage[dvipsnames]{xcolor}
\usepackage{graphicx,pifont}
\usepackage[utopia]{mathdesign}
\usepackage[a4paper,top=3cm,bottom=3cm,left=3.3cm,right=3.3cm,marginparwidth=60pt]{geometry}
\usepackage[utf8]{inputenc}
\usepackage{braket,caption,comment,mathtools,stmaryrd,enumitem,amsmath,float,soul}
\usepackage[colorinlistoftodos]{todonotes} 
\usepackage[usestackEOL]{stackengine}
\usepackage{multirow,booktabs,microtype,relsize}
\usepackage[foot]{amsaddr}
\usepackage[colorlinks,bookmarks]{hyperref} %
      \hypersetup{colorlinks,%
            citecolor=britishracinggreen,%
            filecolor=black,%
            linkcolor=cobalt,%
            urlcolor=cornellred}
      \numberwithin{equation}{section}
\usepackage[capitalise]{cleveref}
\DeclareSymbolFont{cmarrows}{OMS}{cmsy}{m}{n}
\SetSymbolFont{cmarrows}{bold}{OMS}{cmsy}{b}{n}
\DeclareMathSymbol{\cmminus}{\mathbin}{cmarrows}{"00}

\DeclareMathSymbol{\leftrightarrow}{\mathrel}{cmarrows}{"24}
\DeclareMathSymbol{\leftarrow}{\mathrel}{cmarrows}{"20}
   
\DeclareMathSymbol{\rightarrow}{\mathrel}{cmarrows}{"21}
   \let\to=\rightarrow
\DeclareMathSymbol{\mapstochar}{\mathrel}{cmarrows}{"37}
   \def\mapsto{\mapstochar\rightarrow}

\DeclareSymbolFont{usualmathcal}{OMS}{cmsy}{m}{n}
\DeclareSymbolFontAlphabet{\mathcal}{usualmathcal}
\DeclareMathAlphabet\BCal{OMS}{cmsy}{b}{n}

\makeatletter
\newcommand{\mylabel}[2]{#2\def\@currentlabel{#2}\label{#1}}
\makeatother

\definecolor{cornellred}{rgb}{0.7, 0.11, 0.11}
\definecolor{britishracinggreen}{rgb}{0.0, 0.26, 0.15}
\definecolor{cobalt}{rgb}{0.0, 0.28, 0.67}

\newcommand{\V}{\mathbf{V}}

\newcommand{\into}{\hookrightarrow}
\newcommand{\onto}{\twoheadrightarrow}
\newcommand{\HH}{\mathrm{H}}
\newcommand{\Def}{\mathsf{Def}}
\newcommand{\OO}{\mathscr O}
\newcommand{\bfk}{\mathbf{k}}
\newcommand{\derived}{\mathbf{D}}

\newcommand{\boldit}[1]{\boldsymbol{#1}}
\newcommand{\pt}{{\mathsf{pt}}}

\DeclareMathOperator{\vd}{vdim}
\DeclareMathOperator{\rk}{rk}
\DeclareMathOperator{\vir}{\mathrm{vir}}

\DeclareMathOperator{\Bun}{Bun}
\DeclareMathOperator{\Ob}{Ob}
\DeclareMathOperator{\Perf}{Perf}
\DeclareMathOperator{\Cone}{\mathrm{Cone}}

\DeclareMathOperator{\RR}{\mathbf{R}}
\DeclareMathOperator{\QQ}{\mathrm{Q}}
\DeclareMathOperator{\LL}{\mathbf{L}}
\DeclareMathOperator{\lHom}{\mathscr{H}\kern-0.3em\mathit{om}}
\DeclareMathOperator{\RRlHom}{\mathbf{R}\kern-0.025em\mathscr{H}\kern-0.3em\mathit{om}}

\DeclareMathOperator{\lExt}{{\mathscr{E}\kern-0.2em\mathit{xt}}}

\DeclareMathOperator{\Sym}{Sym}
\DeclareMathOperator{\Var}{Var}

\DeclareMathOperator{\Coh}{Coh}

\DeclareMathOperator{\Hom}{Hom}

\DeclareMathOperator{\Ext}{Ext}

\DeclareMathOperator{\Quot}{Quot}
\DeclareMathOperator{\Spec}{Spec}

\DeclareMathOperator{\Art}{Art}
\DeclareMathOperator{\Sets}{Sets}

\DeclareMathOperator{\Sch}{Sch}

\newcommand{\TT}{\mathbf{T}}
\newcommand{\BA}{{\mathbb{A}}}

\newcommand{\BC}{{\mathbb{C}}}

\newcommand{\BE}{{\mathbb{E}}}
\newcommand{\BF}{{\mathbb{F}}}
\newcommand{\BG}{{\mathbb{G}}}

\newcommand{\BL}{{\mathbb{L}}}

\newcommand{\BP}{{\mathbb{P}}}
\newcommand{\BQ}{{\mathbb{Q}}}

\newcommand{\BZ}{{\mathbb{Z}}}

\newcommand{\CI}{{\mathcal I}}
\newcommand{\CJ}{{\mathcal J}}
\newcommand{\CK}{{\mathcal K}}

\newcommand{\CT}{{\mathcal T}}
\newcommand{\CU}{{\mathcal U}}

\newcommand{\CW}{{\mathcal W}}

\newcommand{\CZ}{{\mathcal Z}}

\usepackage[all]{xy}
\usepackage{tikz}
\usepackage{tikz-cd}
\usepackage{adjustbox}
\usepackage{rotating}
\usepackage{comment}
\newcommand*{\isoarrow}[1]{\arrow[#1,"\rotatebox{90}{\(\sim\)}"
]}
\usetikzlibrary{matrix,shapes,intersections,arrows,decorations.pathmorphing,decorations.markings}
\tikzset{commutative diagrams/.cd,
mysymbol/.style={start anchor=center,end anchor=center,draw=none}}
\newcommand\MySymb[2][\square]{%
  \arrow[mysymbol]{#2}[description]{#1}}
\tikzset{
shift up/.style={
to path={([yshift=#1]\tikztostart.east) -- ([yshift=#1]\tikztotarget.west) \tikztonodes}
}
}

\usepackage{youngtab} 

\usepackage{ytableau} 



\theoremstyle{definition}

\newtheorem*{lemma*}{Lemma}
\newtheorem*{theorem*}{Theorem}
\newtheorem*{example*}{Example}
\newtheorem*{fact*}{Fact}
\newtheorem*{notation*}{Notation}
\newtheorem*{definition*}{Definition}
\newtheorem*{prop*}{Proposition}
\newtheorem*{remark*}{Remark}
\newtheorem*{corollary*}{Corollary}

\newtheorem*{conventions*}{Conventions}

\newtheorem{definition}{Definition}[section]

\newtheorem{example}[definition]{Example}

\newtheorem{remark}[definition]{Remark}

\newtheoremstyle{thm} 
        {3mm}
        {3mm}
        {\slshape}
        {0mm}
        {\bfseries}
        {.}
        {1mm}
        {}
\theoremstyle{thm}
\newtheorem{theorem}[definition]{Theorem}
\newtheorem{corollary}[definition]{Corollary}
\newtheorem{lemma}[definition]{Lemma}
\newtheorem{prop}[definition]{Proposition}

\newtheorem{thm}{Theorem}

\newtheorem*{Acknowledgments*}{Acknowledgments}

\newcommand{\ns}{\boldit{s}}
\newcommand{\nd}{\boldit{d}}

\newcommand{\nr}{\boldit{r}}
\newcommand{\nn}{\boldit{n}}
\newcommand{\nq}{\boldit{q}}

\newcommand{\nv}{\boldit{v}}
\newcommand{\Flag}{\mathrm{Flag}}
\newcommand{\fix}{\mathsf{fix}}
\newcommand{\ch}{\mathsf{ch}}

\title[Hyperquot schemes on curves]{Hyperquot schemes on curves: \\ virtual class and motivic invariants}

\author{Sergej Monavari and Andrea T. Ricolfi}

\begin{document}
\begin{abstract}
Let $C$ be a smooth projective curve, $E$ a locally free sheaf. Hyperquot schemes on $C$ parametrise flags of coherent quotients of $E$ with fixed Hilbert polynomial, and offer alternative compactifications to the spaces of maps from $C$ to partial flag varieties. Motivated by enumerative geometry, in this paper we construct a perfect obstruction theory (and hence a virtual class and a virtual structure sheaf) on these moduli spaces, which we use to provide criteria for smoothness and unobstructedness. Under these assumptions, we determine their motivic partition function in the Grothendieck ring of varieties, in terms of the motivic zeta function of $C$.
\end{abstract}
\maketitle

{\hypersetup{linkcolor=black}\tableofcontents}

\section{Introduction}
\subsection{Overview}
Let $C$ be a smooth projective curve over an algebraically closed field $\bfk$, and let $E$ be a locally free sheaf of rank $r$ on $C$. Fix a positive integer $l$, along with two $l$-tuples of integers $\ns=(s_1,\ldots,s_l)$ and $\nd=(d_1,\ldots,d_l)$, where $0\leqslant s_1\leqslant \cdots\leqslant s_l\leqslant r$. In this paper we study a generalisation of the classical Grothendieck's Quot scheme \cite{Grothendieck_Quot}, namely the \emph{hyperquot scheme}
\[
\Quot_C(E,\ns,\nd),
\]
parametrising sequences
\[
\begin{tikzcd}
K_l\arrow[hook]{r} & \cdots\arrow[hook]{r} & K_1\arrow[hook]{r} & E\arrow[two heads]{r} & T_l\arrow[two heads]{r} & \cdots\arrow[two heads]{r} & T_1
\end{tikzcd}
\]
in $\Coh(C)$, where $K_i = \ker (E \onto T_i)$ and $(\rk T_i,\deg T_i) = (s_i,d_i)$ for all $i=1,\ldots,l$.

When $E=\OO_C^{\oplus r}$, the hyperquot scheme provides a natural compactification of the space of morphisms $M_{\nd}(C,\Flag(\ns, r))$, of multidegree $\nd$, from $C$ to the partial flag variety $\Flag(\ns,r)$. This compactification is a natural alternative to the moduli space of \emph{stable maps} of Kontsevich, and was effectively exploited to compute the quantum cohomology ring of partial flag varieties, see e.g.~\cite{CF_quantum_duke, CF_quantum_cohom_IMRN, CF_quantum_Trans, Chen_quantum, Ber_quantum_shub} for a non-exhaustive list of references.

\smallbreak
It is natural to approach the study of a moduli space by examining its \emph{invariants}. In this paper we pursue this direction in two different ways:

\begin{itemize}
\item [\mylabel{objective-1}{(1)}] We construct a perfect obstruction theory in the sense of Behrend--Fantechi \cite{BF_normal_cone}, leading to a virtual fundamental class and a virtual structure sheaf, on the space $\Quot_C(E,\ns,\nd)$. This paves the way to the study of the virtual intersection theory of the hyperquot scheme.
\item [\mylabel{objective-2}{(2)}] We compute the motivic class of $\Quot_C(E,\ns,\nd)$ in the Grothendieck ring of varieties $K_0(\Var_{\bfk})$, whenever $E$ is split and the space is smooth and unobstructed. 
\end{itemize}

The ring valued invariants that `behave like' the topological Euler characteristic are called \emph{motivic}. By the well-known universal property of $K_0(\Var_{\bfk})$, our results compute all motivic invariants of the  spaces $\Quot_C(E,\ns,\nd)$ under the smooth and unobstructed assumption.

We note that \ref{objective-2} generalises Chen's results \cite{Chen_hyperquot} on the Poincar\'e polynomials of the hyperquot scheme attached to $(C,E)=(\BP^1,\OO^{\oplus r})$. In particular, we discover new instances of smooth hyperquot schemes. 

In the next subsection, we state our results in greater detail.

\subsection{Statements of the main results}
Our first main result is the existence of a perfect obstruction theory on the hyperquot scheme of a curve.

\begin{thm}[\Cref{thm: pot 2 nest}, \Cref{cor: virtual cycles}]\label{mainthm-POT}
Let $C$ be a smooth projective curve of genus $g$, $E$ a locally free sheaf of rank $r>0$, and $l$ a positive integer. Fix tuples $\ns=(s_1\leqslant \dots \leqslant s_l)$ and $\nd=(d_1, \dots, d_l)$. The hyperquot scheme $\QQ = \Quot_C(E,\ns,\nd)$ admits a perfect obstruction theory
\[
\begin{tikzcd}
\BE_{\QQ} = \Cone\Biggl(\displaystyle\bigoplus_{i=1}^l \RRlHom_\pi (\CK_i,\CK_i) \arrow{r}{} & \displaystyle\bigoplus_{i=0}^{l-1} \RRlHom_\pi (\CK_{i+1},\CK_i)\Biggr)^\vee \arrow{r} & \BL_{\QQ},
\end{tikzcd}
\]
where we set $\CK_0 = E_{\QQ}$. In particular, $\QQ$ admits a virtual fundamental class and a virtual structure sheaf 
\[
[\QQ]^{\vir} \in A_{\vd}(\QQ), \qquad \OO^{\vir}_{\QQ} \in K_0(\QQ).
\]
Setting $s_{l+1}=r$ and $s_0=0$, the virtual class lives in dimension
\[
\vd =  (1-g)\cdot \dim \Flag(\ns,r)+ \sum_{i=1}^l d_i (s_{i+1}-s_{i-1})-\deg E\cdot s_l.
\]
\end{thm}

In the statement, $\Flag(\ns,r)$ refers to the partial flag variety (cf.~\Cref{sec:partial-flag}), whose dimension is $\sum_{i=1}^ls_i(s_{i+1}-s_i)$, as recalled in \eqref{dim-partial-flag}. It can be seen as a hyperquot scheme over a point, cf.~\Cref{example:hyperquots}.

The perfect obstruction theory of \Cref{mainthm-POT} allows one to study the classical deformation theory of points in the hyperquot scheme (cf.~\Cref{cor:tangent-obs}). We give in \Cref{prop: smooth l=1} a sufficient condition for smoothness (and unobstructedness) in the unnested case. In the nested case, we obtain the following result, generalising the known case of the trivial bundle over $\BP^1$.

\begin{thm}[{\Cref{prop: smooth nested flag}}]
Let $L$ be any line bundle on $\BP^1$ and $E=L\otimes \left(\OO^{\oplus n}\oplus \OO(1)^{\oplus m}\right)$ for some $n, m$. Then the hyperquot scheme $\Quot_{\BP^1}(E, \ns, \nd) $ is smooth and unobstructed for all $\nd$ and for all $\ns$.
\end{thm}

Assume that $E=\bigoplus_{\alpha=1}^r L_\alpha$ splits into a sum of line bundles on $C$. Then $\Quot_C(E,\ns, \nd)$ admits the action of the algebraic torus $\TT=\BG_m^r$ as in \cite{Bifet}. We provide in \Cref{sec: comb fixed} a combinatorial description of the fixed locus, indexing it by permutations $\sigma \in P_{\boldit{r}}$ (cf.~\Cref{sec: l geq 1}). We are then in a position to apply the Białynicki-Birula decomposition to the hyperquot scheme, provided it is smooth. This allows us to determine a closed formula for the motivic generating function
\[
\mathsf{Z}_{C,E, \ns}(\nq)=\sum_{\nd}\,\bigl[\Quot_C(E, \ns, \nd)\bigr]\boldit{q}^{\nd}\in K_0(\Var_{\bfk})(\!( q_1, \dots, q_l)\!),
\]
in terms of the motivic zeta function
\[
\zeta_C(t) = \sum_{n\geqslant 0}\,[C^{(n)}]t^n
\]
of the curve and two combinatorial exponents $h^\sigma_0$ and $h^\sigma_2$, defined in \eqref{eqn:h_0,h_2}.
The result is the following.

\begin{thm}[{\Cref{thm: theorem motivic all g}, \Cref{thm: fact genus 0}}]\label{thm:motivic-partition-function}
Let $C$ be a smooth projective curve of genus $g$ and $E=\bigoplus_{\alpha=1}^rL_\alpha$ a split locally free sheaf such that the hyperquot scheme $\Quot_C(E, \ns, \nd)$ is smooth and unobstructed. Set $r_i = r-s_i$ for $i=1,\ldots,l$ and $r_{l+1}=0$. There is an identity
\begin{multline*}
\mathsf{Z}_{C,E, \ns}(\nq)=\sum_{\sigma\in P_{ \nr}}\prod_{j=1}^{l}\left(\BL^{h^\sigma_{2}(j,g, \deg L_1,\dots,  \deg L_r)} q_j^{ \sum_{\alpha=r_j+1}^{r}\deg L_{\sigma(\alpha)}}\right) \\
\cdot \prod_{1\leqslant i\leqslant j\leqslant l}\prod_{\alpha=r_{j+1}+1}^{r_j}\zeta_C\left(\BL^{ h^\sigma_0(i,\alpha)+r_i-r_j+\alpha-r_{j+1}-1}q_i\cdots q_j \right).
\end{multline*}
In particular, $\mathsf{Z}_{C,E, \ns}(\nq)$ is a rational function. 

For the pair $(C,E) = (\BP^1,\OO^{\oplus r})$, one has the identity
\[
\mathsf{Z}_{\BP^1,\OO^{\oplus r}, \ns}(\nq)=  [\Flag(\ns,r)] \prod_{1\leqslant i\leqslant j \leqslant l}\prod_{\alpha=r_{j+1}+1}^{r_j}\frac{1}{\left(1-\BL^{r_i-\alpha} q_i\cdots q_j \right)\left(1-\BL^{r_{i-1}-\alpha+1} q_i\cdots q_j \right)}.
\]
\end{thm}
Our formulas fit nicely in the philosophical idea that \emph{rank $r$ theories are determined by the rank $1$ theory}. Up to now, this phenomenon has been observed in several places, see for instance the motivic computations \cite{MR_nested_Quot, CR_framed_motivic, CRR_higher_rank}  and the recent results in  Donaldson--Thomas theory \cite{FMR_higher_rank, fasola2023tetrahedron, FT_0, FT_1}  and  String Theory \cite{NP_colors, dZNPZ_playing_index_M_theory, Cir_M2_index}.

We also notice that the computation of the motivic partition functions in \Cref{thm:motivic-partition-function} allows one to spot the connectedness of the moduli space, by first specialising to Poincar\'e polynomials $\mathsf P(\Quot_C(E,\ns,\nd),q)$ and then extracting the constant coefficient, which coincides with $b_0(\Quot_C(E,\ns,\nd))$. This is carried out in \Cref{cor:connectedness} for $(C,E) = (\BP^1,\OO^{\oplus r})$.

At the level of Euler characteristics (over the complex field $\bfk=\BC$), we have a complete description of the generating function
\[
\mathsf{Z}^{\mathrm{top}}_{C,E, \ns}(\nq)=\sum_{\nd}\,e_{\mathrm{top}}(\Quot_C(E, \ns, \nd))\boldit{q}^{\nd}\in \BZ(\!( q_1, \dots, q_l)\!),
\]
without any smoothness assumption.

\begin{thm}[{\Cref{thm:euler-top}}]
Let $C$ be a smooth projective complex curve of genus $g$ and $E=\bigoplus_{\alpha=1}^rL_\alpha$ a split locally free sheaf. There is an identity
\begin{align*}
\mathsf{Z}^{\mathrm{top}}_{C,E, \ns}(\nq)&=\left(\sum_{\sigma\in P_{ \nr}}\prod_{j=1}^{l} q_j^{ \sum_{\alpha=r_j+1}^{r}\deg L_{\sigma(\alpha)}}\right)\cdot\prod_{1\leqslant i\leqslant j\leqslant l}(1-q_i\cdots q_j)^{(2g-2)(r_j-r_{j+1})}.
    \end{align*}
\end{thm}

\subsection{Relation with Enumerative Geometry}\label{sec:existing-work}
We conclude the introduction by pointing out some interesting applications of our results, and how they should be considered in the more general framework of enumerative geometry.

\subsubsection{Virtual invariants}
There has been a recent interest in defining and studying virtual invariants of Quot schemes on varieties of various dimensions, usually motivated by parallel computations in topological string theory, see e.g.~\cite{FMR_higher_rank, fasola2023tetrahedron,Lim_more_vir_sur, Lim_vir, bojko2022wallcrossing}. It would be of particular interest to extend to the hyperquot setting the \emph{Vafa--Intriligator formulas} proved in \cite{MO_duke} --- see also \cite{Sin} for an isotropic version of the latter --- and the tautological integrals, studied e.g.~in \cite{OP_quot_schemes_curves_surfaces, OS_taut}. As a first step of this program, we prove in Appendix \ref{app:functoriality-VFC} some functorial properties of the virtual fundamental class, and in Appendix \ref{app: chi_y} we discuss some structural properties of the virtual $\chi_{-y}$-genus, leaving more general computations for future work.
Finally, it is worth mentioning that the hyperquot scheme attached to a curve $C$ with trivial bundle $E=\OO_C^{\oplus r}$ coincides with the (markings-free) quasimap compactification of the space of maps to the partial flag variety, see \cite[Sec.~7.2]{zbMATH06244048} and in particular Example 7.2.3 there.\footnote{We thank I. Ciocan-Fontanine for communicating this to us.} Moreover, in loc.~cit.~it is shown that the moduli stack of quasimaps admits a perfect obstruction theory in greater generality.
\subsubsection{Higher dimensional varieties}
Moduli spaces of sheaves with nesting conditions should not be regarded as mere variations of classical moduli spaces, as they  appear rather  naturally while studying moduli spaces of sheaves on higher dimensional \emph{local} varieties, e.g.~vector bundles on curves and surfaces. The prototypical situation is the following: let $Y$ be a projective variety, and $X$ be the total space of a (split) locally free sheaf on $Y$. Then $X$ admits a natural torus action scaling the fibers, which lifts to the moduli spaces of sheaves on $X$ \cite{Ric_equivariant_Atiyah}. The fixed locus of this torus action  can be usually identified with another moduli space of sheaves on $Y$, satisfying suitable nesting conditions. Therefore, by localisation techniques, enumerative invariants of moduli spaces of sheaves on  $X$ are reduced to enumerative invariants on moduli spaces of \emph{nested} sheaves on $Y$. See for instance applications of this idea for virtual invariants in Donaldson--Thomas and Vafa--Witten theory \cite{Mon_double_nested, KT_vertical, GSY_local, Laa} and for motivic invariants of moduli spaces of Higgs bundles \cite{HGPM_motives_chains, HL_motives_Higgs, HL_formula_motives}.

\subsubsection{Cohomology}
Recently Marian--Neguţ  \cite{marian2023cohomology} described an action of the shifted Yangian of $\mathfrak{sl}_2$ on  the cohomology groups of the Quot schemes of 0-dimensional quotients on a smooth projective curve. It  would be interesting to understand how to generalise this action to our setting with nested  and possibly 1-dimensional quotients, see also \cite{Mochizuki_filt_scheme}. 

\subsection{Conventions}
All \emph{schemes} are of finite type over a fixed algebraically closed field $\bfk$. A \emph{variety} is a reduced separated $\bfk$-scheme. If $Y$ is a scheme and $Y_1,\ldots,Y_s$ are locally closed subschemes of $Y$, we say that they form a (locally closed) \emph{stratification}, denoted `$Y=Y_1\amalg\cdots\amalg Y_s$' with a slight abuse of notation, if the natural morphism of schemes $Y_1\amalg\cdots\amalg Y_s \to Y$ is bijective. We denote by $\derived^{[a,b]}(Y)$ the full subcategory of the derived category of coherent sheaves $\derived(Y)$, consisting of complexes having (possibly) nonvanishing cohomology in the interval $[a,b]$. We use the same notation for the full subcategory $\Perf^{[a,b]}(Y)$ of  the triangulated category $\Perf(Y)$ of perfect complexes on $Y$. We denote the $\OO_Y$-linear dual $\lHom_{\OO_Y}(F,\OO_Y)$ of a coherent sheaf $F\in \Coh(Y)$ by $F^\ast$, reserving the notation $(-)^\vee$ for the derived dual $\RRlHom_{\OO_Y}(-,\OO_Y)$. Given a morphism $\pi\colon Y \to B$, we set $\RRlHom_\pi(-,-) = \RR\pi_\ast \RRlHom(-,-)$. We sometimes write derived pullbacks $\mathbf{L}\pi^\ast$ simply by $\pi^\ast$.

\subsection*{Acknowledgements}
We thank Alina Marian for bringing to our attention Chen's work \cite{Chen_hyperquot}. We are grateful to Shubham Sinha for pointing out \Cref{lemma:POT-on-abelian-cone} to us. We thank Ionu\c{t} Ciocan-Fontanine, Barbara Fantechi, Alina Marian and Shubham Sinha for very helpful discussions. S. M. thanks Eric Chen and Denis Nesterov for useful discussions. S.M. is supported by the Chair or Arithmetic Geometry, EPFL. Both authors are members of GNSAGA of INDAM. 
\section{Hyperquot scheme}
\subsection{Partial flag varieties}\label
{sec:partial-flag}
Fix integers $r,l>0$ and a nondecreasing $l$-tuple of nonnegative integers $\ns=(s_1\leqslant \dots \leqslant s_l)$ such that $s_l\leqslant r$. To ease the notation, we set $s_0=0$, $s_{l+1}=r$ and $r_i=r-s_i$ for all $i=0, \dots, l+1$.

We denote by  $\Flag(\ns,r)$   the \emph{partial flag variety}, namely the scheme parametrising flags of vector spaces
\[
V_l\subset \dots \subset V_1 \subset V,
\]
where $V$ is a fixed  $r$-dimensional $\bfk$-vector space and $\dim V_i =r_i$. It is well-known that $ \Flag(\ns,r)$ is a smooth projective variety of dimension
\begin{equation}\label{dim-partial-flag}
    \dim \Flag(\ns,r)=\sum_{i=1}^l s_i(s_{i+1}-s_i).
\end{equation}
The case $l=1$ yields the Grassmannian $\Flag(s_1,r) = G(r_1,V)$. The \emph{hyperquot scheme}, introduced in the next subsection, can be seen as a vaste generalisation of the partial flag variety.

\subsection{The hyperquot scheme}
Let $Y$ be a projective variety, $E$ a coherent sheaf on $Y$. Fix an integer $l>0$ and a vector
\[
\nv = (v_1,\ldots,v_l) \in \HH^\ast (Y,\BQ)^{\oplus l}
\]
of Chern characters $v_i \in \HH^\ast(Y,\BQ)$. We define the \emph{ hyperquot functor} associated to $(Y,E,\nv)$ to be the functor $\mathsf{Quot}_Y(E,\nv) \colon \Sch_{\bfk}^{\mathrm{op}} \to \Sets$ sending a $\bfk$-scheme $B$ to the set of isomorphism classes of subsequent quotients
\[
\begin{tikzcd}
E_B\arrow[two heads]{r} & \CT_l\arrow[two heads]{r} & \cdots\arrow[two heads]{r} & \CT_1
\end{tikzcd}
\]
where $E_B$ is the pullback of $E$ along $Y\times B \to Y$ and $\CT_i \in \Coh(Y\times B)$ is a $B$-flat family of coherent sheaves such that $\CT_i|_{Y \times b} \in \Coh(Y)$ has Chern character $v_i$ for all $b \in B$ and $i=1,\ldots,l$. Two `nested quotients' 
\[
\begin{tikzcd}
E_B\arrow[two heads]{r} & \CT_l\arrow[two heads]{r} & \cdots\arrow[two heads]{r} & \CT_1,
& E_B\arrow[two heads]{r} & \CT_l'\arrow[two heads]{r} & \cdots\arrow[two heads]{r} & \CT_1'
\end{tikzcd}
\]
are considered isomorphic when $\ker (E_B \onto \CT_i) = \ker (E_B \onto \CT_i')$ for all $i=1,\ldots,l$.

The representability of the functor $\mathsf{Quot}_Y(E,\nv)$ can be proved adapting the proof of \cite[Thm.~4.5.1]{Ser_deformation} or by an explicit induction on $l$ as in \cite[Sec.~2.A.1]{modulisheaves}. We define 
\[
\Quot_Y(E, \nv)
\]
to be the moduli scheme representing the hyperquot functor. We call it the \emph{hyperquot scheme}. Its closed points are then in bijection with the set of isomorphism classes of subsequent quotients
\begin{equation}\label{eqn:nested-quotients}
\begin{tikzcd}
E\arrow[two heads]{r} & T_l\arrow[two heads]{r} & \cdots\arrow[two heads]{r} & T_1
\end{tikzcd}
\end{equation}
where each $T_i \in \Coh(Y)$ is a quotient of $E$ with $\ch(T_i) = v_i$. Equivalently, it parametrises sequences
\begin{equation}\label{point=sequence}
\begin{tikzcd}
K_l\arrow[hook]{r} & \cdots\arrow[hook]{r} & K_1\arrow[hook]{r} & E\arrow[two heads]{r} & T_l\arrow[two heads]{r} & \cdots\arrow[two heads]{r} & T_1
\end{tikzcd}
\end{equation}
where we have set $K_i = \ker (E \onto T_i)$.
The  hyperquot scheme comes with a closed immersion, inside a product of ordinary Quot schemes
\begin{equation}\label{eqn:embedding_in_product}
\begin{tikzcd}
\Quot_Y(E,\nv) \arrow[hook]{r} & \displaystyle\prod_{i=1}^l \Quot_Y(E,v_i),
\end{tikzcd}
\end{equation}
cut out by the nesting condition $K_l \into \cdots \into K_1$. In particular, it is a projective scheme. 

\begin{example}\label{example:hyperquots}
The following examples recover well-known moduli spaces.
\begin{enumerate}
    \item Fix $r \in \BZ_{>0}$. If $Y=\Spec \bfk=\pt$ and $\nv=(s_1,\ldots,s_l)$ is a nondecreasing tuple of nonnegative integers with $s_l\leqslant r$, we have $\Quot_{\pt}(\bfk^{\oplus r},\nv) = \Flag(\ns,r)$, the partial flag variety of \Cref{sec:partial-flag}.
    \item If $E=L$ is locally free of rank $1$ and $v_i = (0,0,\ldots,n_i)$ with $n_l\geqslant \cdots\geqslant  n_1\geqslant 0$, then $\Quot_Y(L,\nv)$ is the nested Hilbert scheme $Y^{[\nn]}$. If $E$ is an arbitrary locally free sheaf, one obtains the nested Quot scheme of points.
\end{enumerate}
\end{example}

\subsection{Tangent space to hyperquot scheme}
Let $z \in \Quot_Y(E,\nv)$ be a closed point, representing a sequence as in \eqref{point=sequence}. There is an inclusion of $\bfk$-vector spaces
\[
\begin{tikzcd}
T_z \Quot_Y(E,\nv) \arrow[hook]{r} & \displaystyle \bigoplus_{i=1}^l \Hom_Y(K_i,T_i),    
\end{tikzcd}
\]
which is just the tangent map to \eqref{eqn:embedding_in_product} at $z$. An $l$-tuple $(\delta_1,\delta_2,\ldots,\delta_l)$ of tangent vectors $\delta_i \in \Hom_Y(K_i,T_i)$ belongs to $T_z \Quot_Y(E,\nv)$ if and only if the diagram
\[
\begin{tikzcd}[row sep=large,column sep=large]
    K_l\arrow[hook]{r}{\iota_{l-1}}\arrow{d}{\delta_l} & K_{l-1} \arrow[hook]{r}{\iota_{l-2}}\arrow{d}{\delta_{l-1}} & \cdots \arrow[hook]{r}{\iota_2} & K_2\arrow[hook]{r}{\iota_{1}}\arrow{d}{\delta_2} & K_1 \arrow{d}{\delta_1} \\
    T_l \arrow[two heads]{r}{p_{l-1}} & T_{l-1} \arrow[two heads]{r}{p_{l-2}} & \cdots \arrow[two heads]{r}{p_2} & T_2 \arrow[two heads]{r}{p_{1}} & T_1
\end{tikzcd}
\]
commutes (see also \cite[Sec. 2]{MR_nested_Quot}). Such commutativity relations,
\[
\delta_i\circ\iota_i-p_i\circ \delta_{i+1} = 0, \quad i=1,\ldots,l-1,
\]
can be encoded in an exact sequence of $\bfk$-linear maps
\begin{equation}\label{ses:tg-space}
\begin{tikzcd}
0 \arrow{r} 
& T_z \Quot_Y(E,\nv) \arrow{r} 
& \displaystyle\bigoplus_{i=1}^l \Hom_Y(K_i,T_i) \arrow{r}{\Delta_z}
& \displaystyle\bigoplus_{i=1}^{l-1} \Hom_Y(K_{i+1},T_i), 
\end{tikzcd}
\end{equation}
where 
\[
\Delta_z(\delta_1,\ldots,\delta_l) = (\delta_i\circ\iota_i-p_i\circ \delta_{i+1})_{1\leqslant i\leqslant l-1}.
\]

\subsection{Our setup: the curve case}\label{sec:setup}
From now on, we work with the hyperquot scheme of a fixed smooth, projective, connected curve $C$ of genus $g$, i.e.~we set $Y=C$. The following data are fixed throughout:
\begin{itemize}
\item [$\circ$] two positive integers $r,l$,
\item [$\circ$] a locally free sheaf $E$ of rank $r$,
\item [$\circ$] a nondecreasing $l$-tuple $\ns=(s_1\leqslant \dots \leqslant s_l)$ of integers $s_i \in \BZ_{\geqslant 0}$ such that $s_l \leqslant r$,
\item [$\circ$] a tuple $\nd=(d_1, \ldots, d_l)$ of integers $d_i \in \BZ$.
\end{itemize}
With this data, we form the vector of Chern characters $\nv = ((s_1,d_1),\ldots,(s_l,d_l)) \in \HH^\ast(C,\BZ)^{\oplus l}$ and we set
\[
\Quot_C(E,\ns,\nd) = \Quot_C(E,\nv).
\]
This hyperquot scheme parametrises sequences \eqref{point=sequence} where $T_i$ is a coherent quotient of $E$ of rank $s_i$ and degree $d_i$. In particular, the kernels $K_i = \ker (E \onto T_i)$ are locally free of rank $r_i = r-s_i$.

\section{Perfect obstruction theory}
\subsection{Perfect obstruction theories}\label{sec: OT}
Let $X$, $B$ be algebraic stacks over $\bfk$, and assume $B$ is smooth. Let $h \colon X \to B$ be a morphism of Deligne--Mumford type (i.e.~such that the diagonal $\Delta_h \colon X \to X \times_BX$ is unramified). In this situation, the relative cotangent complex $L_h$ defined by Illusie \cite{Illusie_Complex_cotangent_1} has cohomology concentrated in nonpositive degrees, i.e.~it belongs to $\derived^{\leqslant 0}(X)$. Let $\BL_h \in \derived^{[-1,0]}(X)$ denote its cut-off at $-1$. 

A \emph{relative perfect obstruction theory} \cite[Sec.~7]{BF_normal_cone} on $h\colon X \to B$ is the datum of a morphism 
\[
\begin{tikzcd}
\BE\arrow{r}{\phi} & \BL_{h}    
\end{tikzcd}
\]
in the derived category $\derived^{[-1,0]}(X)$, where $\BE$ is a perfect complex of perfect amplitude contained in $[-1,0]$, such that
\begin{itemize}
    \item [$\circ$] $h^0(\phi)$ is an isomorphism,
    \item [$\circ$] $h^{-1}(\phi)$ is surjective.
\end{itemize}
If $B=\Spec \bfk$, we say $\phi$ is an \emph{absolute} perfect obstruction theory. 

The dual $\BE^\vee \in \derived^{[0,1]}(X)$ is called the \emph{virtual tangent bundle} of the obstruction theory, and its first cohomology $\Ob_\phi = h^1(\BE^\vee)$ is called the \emph{obstruction sheaf} attached to $(\BE,\phi)$.
The \emph{virtual dimension} of $X$ with respect to $(\BE,\phi)$ is the integer $\vd = \rk \BE + \dim B$. 
By \cite{BF_normal_cone},  a perfect obstruction theory induces a  \emph{virtual fundamental class} 
\[
[X]^{\vir} \in A_{\vd}(X)
\]
and, if $B=\Spec \bfk$, a \emph{virtual structure sheaf} 
\[
\OO_X^{\vir} \in K_0(X).
\]
By a result of Siebert \cite{Siebert} (resp.~Thomas \cite{Tho_K-theo_Fulton}) the virtual fundamental class in the absolute case $B=\Spec \bfk$ (resp.~the virtual structure sheaf) depends only on the $K$-theory class of $\BE$.
Moreover, Behrend--Fantechi prove \cite[Prop.~7.3]{BF_normal_cone} that if $h^{-1}(\BE)=0$ then $h\colon X \to B$ is smooth, and in this case $[X]^{\vir} = [X]$. Furthermore, if $h\colon X \to B$ is smooth, then $\Ob_\phi$ is locally free and $[X]^{\vir} = e(\Ob_\phi)\cap [X]$, where $e(-)$ denotes the Euler class.

\smallbreak
We now recall two basic facts that will be needed in the next subsection. 

First of all, given a smooth algebraic stack $B$ and a morphism $h\colon X \to B$ with a relative perfect obstruction theory $\phi_h\colon \BE_h \to \BL_h$, the shifted cone
\[
\BE = \Cone(\BE_h \to h^\ast \Omega_B[1])[-1]
\]
is an absolute perfect obstruction theory for $X$.
This can be shown via the diagram
\begin{equation}\label{diag:relative-POT-to-absolute}
\begin{tikzcd}[column sep=large]
h^\ast \Omega_B \arrow[dotted]{r}\arrow[equal]{d} & \BE \arrow[dotted]{d}\arrow[dotted]{r} & \BE_h \arrow{d}{\phi_h}\arrow{r}{\varepsilon\circ \phi_h} & h^\ast \Omega_B[1]\arrow[equal]{d} \\  
h^\ast \Omega_B \arrow{r} & \BL_X\arrow{r} & \BL_h \arrow{r}{\varepsilon} & h^\ast \Omega_B[1]
\end{tikzcd}
\end{equation}
where the dotted arrows can be completed to solid ones, to form a morphism of exact triangles, by the axioms of the derived category, and $\BE \to \BL_X$ turns out to be a perfect obstruction theory thanks to a simple computation in  the long exact sequence in cohomology for the two triangles involved. See also \cite[Constr.~3.13]{Manolache-virtual-pb} for a more general setup.

The second fact we need is the following lemma. 

\begin{lemma}[{\cite[Lemma 2.8]{Sin} and \cite{Sca_pot}}]\label{lemma:POT-on-abelian-cone}
Let $p\colon B' \to B$ be a flat, projective, $1$-dimensional family of algebraic stacks. Let $F \in \Coh(B')$ be a $B$-flat coherent sheaf. Consider the abelian cone 
\[
\begin{tikzcd}
\mathbf S = \Spec_{\OO_B} \Sym (\RR^1 p_\ast F) \arrow{r}{\eta} & B.
\end{tikzcd}
\]
Let $\overline p \colon \mathbf S \times_BB' \to \mathbf S$ and $\overline \eta \colon \mathbf S \times_BB' \to B'$ be the projections. Then there is a relative perfect obstruction theory
\[
\begin{tikzcd}
\RR \overline p_\ast \overline F[1] \arrow{r} & \BL_\eta,
\end{tikzcd}
\]
where $\overline F = \overline \eta^\ast F$.
\end{lemma}

\subsection{Perfect obstruction theory on hyperquot}
Set $\QQ = \Quot_C(E,\ns, \nd)$, where we have fixed $\ns=(s_1\leqslant \dots \leqslant s_l)$ and $\nd = (d_1,\dots, d_l)$ as in \Cref{sec:setup}. Let $\pi\colon \QQ \times C \to \QQ$ denote the projection. We prove in this section that $\QQ$ admits a perfect obstruction theory.

Over $\QQ\times C$, the universal kernels realise a filtration
\begin{equation}\label{universal-filtration}
\begin{tikzcd}
\CK_l \arrow[hook]{r}{\iota_{l-1}} & \CK_{l-1}\arrow[hook]{r}{\iota_{l-2}} & \cdots \arrow[hook]{r}{\iota_2} & \CK_2 \arrow[hook]{r}{\iota_1} & \CK_1 \arrow[hook]{r}{\iota_0} & \CK_0 = E_{\QQ}.
\end{tikzcd}
\end{equation}

\begin{theorem}\label{thm: pot 2 nest}
The hyperquot scheme $\QQ = \Quot_C(E,\ns,\nd)$ admits a perfect obstruction theory
\[
\begin{tikzcd}
\BE_{\QQ} = \Cone\Biggl(\displaystyle\bigoplus_{i=1}^l \RRlHom_\pi (\CK_i,\CK_i) \arrow{r}{} & \displaystyle\bigoplus_{i=0}^{l-1} \RRlHom_\pi (\CK_{i+1},\CK_i)\Biggr)^\vee \arrow{r} & \BL_{\QQ}
\end{tikzcd}
\]
\end{theorem}
\begin{proof}
We prove the case $l=2$. The case $l>2$ is identical and the case $l=1$ is well-known, but also a special case of the case $l=2$ (cf.~\Cref{rmk:no-nesting-POT}).    

So set $l=2$. We need to fix some notation. For $i=1,2$, set $\QQ_i = \Quot_C(E,s_i,d_i)$. Set $B = B_1 \times B_2$ where $B_i = \Bun_C(r-s_i,\deg E-d_i)$ is the (smooth) algebraic stack of vector bundles on $C$, with the indicated Chern character. For $i=1,2$, we have universal bundles $\CU_i \in \Coh(B_i \times C)$ and we denote by 
\[
\begin{tikzcd}
    B_i \times C \arrow{r}{q_i} & B_i, &
    B \times C \arrow{r}{\rho_i} & B_i, & 
    B \times C \arrow{r}{p} & B
\end{tikzcd}
\]
the projections. Consider the map
\[
\begin{tikzcd}
\QQ \arrow{r}{h} & B
\end{tikzcd}
\]
sending $[K_2 \into K_1 \into E \onto T_2 \onto T_1] \mapsto ([K_1],[K_2])$. We have
\[
\Omega_B = \RRlHom_{q_1}(\CU_1,\CU_1)^\vee[-1]\boxplus \RRlHom_{q_2}(\CU_2,\CU_2)^\vee[-1]
\]
and therefore
\[
h^\ast \Omega_B[1] = \RRlHom_{\pi}(\CK_1,\CK_1)^\vee\oplus \RRlHom_{\pi}(\CK_2,\CK_2)^\vee.
\]
Define the sheaf
\begin{align*}
F
&=\rho_1^\ast (\CU_1 \otimes E_{B_1}^\ast \otimes \omega_{q_1}) \oplus \rho_2^\ast \CU_2 \otimes \rho_1^\ast \CU_1^\ast \otimes \rho_2^\ast \omega_{q_2} \\
&=\rho_1^\ast \CU_1 \otimes E_{B}^\ast \otimes \omega_p \oplus \rho_2^\ast \CU_2 \otimes \rho_1^\ast \CU_1^\ast \otimes\omega_p
\end{align*}
in $\Coh(B \times C)$. Then $\QQ$ is open\footnote{This is because the locus in $\Hom(K_1,E)\oplus \Hom(K_2,K_1) = \HH^0(C,K_1^\ast \otimes E)\oplus \HH^0(C,K_2^\ast\otimes K_1) = \HH^1(C,K_1\otimes E^\ast \otimes \omega_C)^\ast \oplus \HH^1(K_2\otimes K_1^\ast\otimes  \omega_C)^\ast$ of pairs of injective maps is open.} in the total space of the abelian cone
\[
\begin{tikzcd}
\mathbf S = \Spec_{\OO_B} \Sym \RR^1 p_\ast F \arrow{r}{\eta} & B,
\end{tikzcd}
\]
and the map $h\colon \QQ \to B$ agrees with $\eta|_{\QQ}\colon \QQ \into \mathbf S \to B$.

We can thus form the diagram 
\begin{equation}\label{big-diagram}
\begin{tikzcd}[row sep=large,column sep=large]
\QQ \times C \MySymb{dr} \arrow[hook]{r}\arrow[swap]{d}{\pi}\arrow[bend left=25]{rr}[description]{h_C} & \mathbf S \times C\MySymb{dr} \arrow{r}{\overline \eta}\arrow{d}{\overline p} & B \times C \arrow{d}{p} \\
\QQ \arrow[bend right=25]{rr}[description]{h}\arrow[hook]{r}{\mathrm{open}} & \mathbf S \arrow{r}{\eta} & B
\end{tikzcd}
\end{equation}
and apply \Cref{lemma:POT-on-abelian-cone} to get a perfect obstruction $\phi_h\colon \BE_h \to \BL_h$ theory relative to $h \colon \QQ \to B$. Using the transitivity triangle for cotangent complexes relative to $h\colon \QQ \to B$, we find a morphism of exact triangles 
\[
\begin{tikzcd}
\BE_{\QQ} \arrow{d}\arrow{r} & \BE_h \arrow{d}{\phi_h}\arrow{r}{\alpha} & \RRlHom_{\pi}(\CK_1,\CK_1)^\vee\oplus \RRlHom_{\pi}(\CK_2,\CK_2)^\vee \arrow[equal]{d} \\  
\BL_{\QQ}\arrow{r} & \BL_h \arrow{r} & h^\ast \Omega_B[1]
\end{tikzcd}
\]
where we have set 
\[
\BE_{\QQ} = \Cone(\alpha)[-1] = \Cone(\alpha^\vee)^\vee.
\]
Note that
\begin{align*}
\BE_{h} 
&= \RR\pi_\ast(h_C^\ast F)[1] \\
&= \RR\pi_\ast(\CK_1 \otimes E_{\QQ}^\ast \otimes \omega_\pi)[1] \oplus \RR\pi_\ast(\CK_2 \otimes \CK_1^\ast \otimes \omega_\pi)[1] \\
&= \RRlHom_\pi(\CK_1,E_{\QQ})^\vee \oplus \RRlHom_\pi(\CK_2,\CK_1)^\vee.
\end{align*}
We used Grothendieck--Verdier duality for the last identity.
This finishes the proof.
\end{proof}

\begin{remark}\label{rmk:no-nesting-POT}
The case $l=2$ also covers the case $l=1$, which corresponds to setting $s_1=d_1=0$. In this case the hyperquot scheme is the classical Quot scheme $\Quot_C(E,s_2,d_2)$ and its obstruction theory is
\[
\begin{tikzcd}
\RRlHom_\pi(\CK_2,\CT_2)^\vee \arrow{r} & \BL_{\Quot_C(E,s_2,d_2)}.
\end{tikzcd}
\]
This has been already constructed by Marian--Oprea \cite{MO_duke} for $E=\OO_C^{\oplus r}$ (but the proof generalises to arbitrary locally free sheaves), Oprea \cite{Oprea_zero_locus}, and Gillam \cite[Lemma~4.7]{Gillam}.
\end{remark}

\begin{remark}\label{rmk:same-K}
The universal quotients $\CT_i \in \Coh(\QQ\times C)$ realise subsequent quotients
\[
\begin{tikzcd}[column sep=large]
E_{\QQ} \arrow[two heads]{r}{p_{l}} & \CT_l \arrow[two heads]{r}{p_{l-1}} & \CT_{l-1}\arrow[two heads]{r}{p_{l-2}} & \cdots \arrow[two heads]{r}{p_2} & \CT_2 \arrow[two heads]{r}{p_1} & \CT_1.
\end{tikzcd}
\]
By functoriality, for each $i=1,\ldots,l-1$ we obtain maps
\[
\begin{tikzcd}[row sep=tiny]
\RRlHom_\pi (\CK_i,\CT_i) \arrow{r}{u_i} & \RRlHom_\pi (\CK_{i+1},\CT_i) & \delta \mapsto \delta\circ \iota_{i}\\
\RRlHom_\pi (\CK_{i+1},\CT_{i+1}) \arrow{r}{v_i} & \RRlHom_\pi (\CK_{i+1},\CT_i) & \delta \mapsto p_{i}\circ \delta
\end{tikzcd}
\]
in $\Perf(\QQ)$, where $\iota_i \colon \CK_i \into \CK_{i-1}$ are the universal inclusions of \eqref{universal-filtration}. We obtain an exact triangle
\begin{equation}\label{eqn:R_piDelta}
\begin{tikzcd}[column sep=large]
\Gamma[-1]\arrow{r} & \displaystyle\bigoplus_{i=1}^l \RRlHom_\pi (\CK_i,\CT_i) \arrow{r}{\RR_\pi\Delta} & \displaystyle\bigoplus_{i=1}^{l-1} \RRlHom_\pi (\CK_{i+1},\CT_i).
\end{tikzcd}
\end{equation}
where $\Gamma = \Cone(\RR_\pi\Delta)$. 
The map $\RR_\pi\Delta$ can be seen as a universal analogue of the map $\Delta_z$ in \eqref{ses:tg-space}. Indeed, the sequence \eqref{ses:tg-space} is obtained from $\RR_\pi\Delta$ by restricting \eqref{eqn:R_piDelta} to $z\in\QQ$ (see \eqref{delta-restricted} below) and taking cohomology.
The complex
\[
\Gamma[-1] \in \derived(\QQ)
\]
is perfect in $[0,1]$, i.e.~it belongs to $\Perf^{[0,1]}(\QQ)$, and has the following two properties:
\begin{itemize}
\item [\mylabel{gamm1-i}{(i)}] It has the same $K$-theory class as $\BE_{\QQ}^\vee$, and
\item [\mylabel{gamm1-ii}{(ii)}] $h^0(\Gamma[-1]) = T_{\QQ} = h^0(\BE_{\QQ}^\vee)$, by the exact sequence \eqref{ses:tg-space}, and because $\BE_{\QQ}$ is a perfect obstruction theory.
\end{itemize}
It follows that $h^1(\Gamma[-1]) = h^1(\BE_{\QQ}^\vee)$ as well.
\end{remark}

Let $z \in \Quot_C(E,\ns,\nd)$ be a closed point representing \eqref{point=sequence}, with inclusion morphism $i_z$. Applying $\LL i_z^\ast$ to \eqref{eqn:R_piDelta} we obtain an exact triangle
\begin{equation}\label{delta-restricted}
\begin{tikzcd}[column sep=large]
\Gamma_z[-1]  \arrow{r} & \displaystyle\bigoplus_{i=1}^l \RR\Hom (K_i,T_i) \arrow{r}{\RR_z\Delta} & \displaystyle\bigoplus_{i=1}^{l-1} \RR\Hom (K_{i+1},T_i),
\end{tikzcd}
\end{equation}
where $\Gamma_z = \Cone(\RR_z\Delta)$.
On the other hand, recall the definition of the deformation functor $\Def_z \subset \mathsf{Quot}_C(E,\ns,\nd)|_{\Art_\bfk}$, sending a local artinian $\bfk$-algebra $A$ (with residue field $\bfk$) to the set of $A$-valued points in $\mathsf{Quot}_C(E,\ns,\nd)(\Spec A)$ restricting to $z$ over the closed point of $\Spec A$.

We then obtain the following corollary of \Cref{thm: pot 2 nest}, which we will use to detect smoothness of hyperquot schemes in \Cref{sec:smoothness}.

\begin{corollary}\label{cor:tangent-obs}
Let $z \in \Quot_C(E,\ns,\nd)$ be a closed point. Then the pair of $\bfk$-vector spaces
\[
T^1_z = h^0(\Gamma_z[-1]),\qquad T^2_z = h^1(\Gamma_z[-1])
\]
defines a tangent-obstruction theory on the deformation functor $\Def_z$.
\end{corollary}

\begin{proof}
Combine \Cref{thm: pot 2 nest} and \Cref{rmk:same-K} with one another.
\end{proof}

\begin{remark}\label{rmk:the-two-cones}
Define
\begin{equation}\label{def:Gamma-cone}
\BF_{\QQ} = \Gamma^\vee[1] = \Cone((\RR_\pi\Delta)^\vee).
\end{equation}
Given the evidence provided in \Cref{rmk:same-K}, we expect an equality of \emph{complexes}
\[
\BE_{\QQ} = \BF_{\QQ}.
\]
We will confirm functoriality properties of the virtual class $[\QQ]^{\vir}$, granting this identity, in \Cref{app:functoriality-VFC}.
\end{remark}

\subsection{Virtual structure}
By the general machinery recalled in Section \ref{sec: OT}, hyperquot schemes on curves admit natural virtual fundamental classes. In this subsection we give a formula for the \emph{virtual dimension} of the perfect obstruction theory, namely the rank of the complex $\BE_{\QQ}$.

Recall that, with the conventions of \Cref{sec:partial-flag}, the dimension of the partial flag variety is
\begin{align*}
    \dim \Flag(\ns,r)=\sum_{i=1}^l s_i(s_{i+1}-s_i).
\end{align*}

The following is a direct consequence of \Cref{thm: pot 2 nest}.

\begin{corollary}\label{cor: virtual cycles}
Let $C$ be a smooth projective curve of genus $g$, $E$ a locally free sheaf of rank $r$,  and fix  tuples $\ns=(s_1\leqslant \dots \leqslant s_l)$ and $\nd=(d_1, \dots, d_l)$.   The hyperquot scheme $\QQ = \Quot_C(E,\ns,\nd)$ admits a virtual fundamental class and a virtual structure sheaf
    \begin{align*}
[\QQ]^{\vir} & \in\, A_{\vd}(\QQ),\\
\OO_{\QQ}^{\vir} &\in\, K_0(\QQ),
\end{align*}
where the virtual dimension is
\[
\vd =   (1-g)\cdot \dim \Flag(\ns,r)+ \sum_{i=1}^l d_i (s_{i+1}-s_{i-1})-\deg E\cdot s_l,
\]
where we set $s_{l+1}=r$ and $s_0=0$.
\end{corollary}

\begin{proof}
  Let
  \[
[K_l\hookrightarrow \dots \hookrightarrow K_1 \hookrightarrow   E \onto T_l \onto \cdots \onto T_1]\in\QQ
\]
be any closed point. Then, thanks to \Cref{rmk:same-K}, the virtual dimension is computed by 
   \[
   \vd = \rk \BE_{\QQ} = \sum_{i=1}^l \rk\RR\Hom (K_{i},T_{i})-\sum_{i=1}^{l-1} \rk\RR\Hom (K_{i+1},T_{i}).
   \]
    By Riemann--Roch, for each $i,j=1, \dots, l$ we have
    \[
    \rk\RR\Hom (K_{j},T_{i})=r_js_i(1-g) + r_j d_i+s_i(d_j-\deg E),
    \]
    which implies
    \begin{align*}
         \vd &= \sum_{i=1}^l\left(r_is_i(1-g) + r_i d_i+s_i(d_i-\deg E)\right)   - \sum_{i=1}^{l-1}\left(r_{i+1}s_i(1-g) + r_{i+1} d_i+s_i(d_{i+1}-\deg E)\right)  \\
         &=  (1-g)\sum_{i=1}^l s_i(s_{i+1}-s_i)-s_l\cdot \deg E +\sum_{i=1}^ls_{i+1}d_{i}-\sum_{i=1}^{l-1}s_id_{i+1}\\
         &= (1-g)\cdot \dim \Flag(\ns,r) -s_l\cdot \deg E+ \sum_{i=1}^l d_i (s_{i+1}-s_{i-1}),
    \end{align*}
as required.
\end{proof}

\begin{remark}
Assume $g>1$, $\deg E=0$ and $d_i=0$ for all $i=1,\ldots,l$. Then $\QQ=\Quot_C(E,\ns,\mathbf{0})$ has negative virtual dimension, namely
\[
\vd = (1-g) \cdot \dim \Flag(\ns,r) < 0.
\]
\end{remark}
\begin{remark}
It was recently proved in \cite{rasul2024irreducibility} that, for $l=2$ and the multidegree $\nd$ large enough, the hyperquot scheme $\QQ=\Quot_C(E,\ns,\mathbf{\nd})$ is irreducible, local complete intersection and that its dimension coincides with our formula for the  virtual dimension, which implies that the virtual fundamental class coincides with the ordinary fundamental class. See also \cite{Popa} for the case of ordinary Quot schemes.
\end{remark}

\section{Smoothness of hyperquot schemes}\label{sec:smoothness}
We exploit the tangent-obstruction theory of \Cref{cor:tangent-obs} to discuss the smoothness of the hyperquot scheme $\Quot_C(E, \ns, \nd)$, where $E$ is a locally free sheaf over a smooth projective curve $C$, as in \Cref{sec:setup}. Notice that if $\rk E=1$, then   $\Quot_C(E, \ns, \nd)$ is isomorphic to a product of symmetric products of $C$,     and therefore it is always smooth and irreducible.

\subsection{The case of 0-dimensional quotients} The smoothness in the case of $0$-dimensional quotients was  obtained in  \cite{MR_nested_Quot, MR_lissite} by a direct computation. We provide here an alternative proof.
\begin{prop}\label{prop: smooth zero dim}
  Let $C$ be a smooth projective curve and $E$ a locally free sheaf on $C$. Then for all $\nd$ the hyperquot scheme $\Quot_C(E, \mathbf{0}, \nd)$ is smooth and unobstructed.
\end{prop}
\begin{proof}
     Let
  \[
z = [K_l\hookrightarrow \dots \hookrightarrow K_1 \hookrightarrow   E \onto T_l \onto \cdots \onto T_1] \in \Quot_C(E, \mathbf{0}, \nd)
\]
be any point and consider the exact sequences
\[
\begin{tikzcd}[row sep=tiny]
0 \arrow{r} & K_i \arrow{r} & K_{i-1} \arrow{r} & Q_i \arrow{r} & 0 \\
0 \arrow{r} & Q_i \arrow{r} & T_i \arrow{r} & T_{i-1} \arrow{r} & 0
\end{tikzcd}
\]
for every $i=1,\dots, l$. Since $Q_i, T_i$ are $0$-dimensional sheaves and $K_i$ is locally free, we obtain the vanishings
\begin{align*}\label{eqn: vanishings}
\Ext^j_C(K_i, T_i)=\Ext^j_C(K_{i+1}, T_i)=\Ext^j_C(K_i, Q_i)=0, \quad j>0,
\end{align*}
which implies that $\Delta_z$, the map introduced in \eqref{ses:tg-space}, is surjective. Therefore, taking the long exact sequence in cohomology attached to \eqref{eqn:R_piDelta} we obtain
\[
\begin{tikzcd}h^1(\Gamma_z[-1])\arrow[hook]{r} &  \displaystyle\bigoplus_{i=1}^l \Ext^1_C(K_i, T_i)=0, 
\end{tikzcd}
\]
  by which we conclude that the obstruction space $h^1(\Gamma_z[-1])$ vanishes as well.
\end{proof}

\subsection{Smoothness in the unnested case} 
We start discussing criteria for smoothness of the ordinary Quot scheme, i.e.~without nesting conditions.
\begin{prop}\label{prop: smooth l=1}
Fix $r\geqslant 2$ and $0<s\leqslant r$. Let $C$ be a smooth projective curve of genus $g$, and let $E=\bigoplus_{\alpha=1}^r L_\alpha$ be a split locally free sheaf on $C$. Assume that
\begin{align}\label{eqn: condition smooth l=1}
\HH^0(C, L_\alpha\otimes L_{\beta}^{-1}\otimes \omega_C)=0, \qquad 1\leqslant \alpha\neq \beta\leqslant r.
\end{align}
Then, for all $d\in \BZ$, the Quot scheme  $  \Quot_C(E, s, d) $ is smooth and unobstructed, and $g\leqslant 1$.
\end{prop}
\begin{proof}
In Section \ref{sec: torus action} we describe a $\TT$-action on $ \Quot_C(E, s, d) $, where $\TT = \BG_m^r$, such that all the torus fixed point are of the form 
\[
\bigoplus_{\alpha=1}^r \,[K_\alpha\hookrightarrow L_\alpha \onto T_\alpha] \in   \Quot_C(E, s, d)^\TT.
\]
For such a torus fixed point, by Serre duality its obstruction space satisfies
\[
\Ext^1_C\left(\bigoplus_{\alpha=1}^r K_\alpha, \bigoplus_{\beta=1}^r T_\beta\right)=\bigoplus_{\alpha, \beta=1}^r \HH^0(C, K_\alpha \otimes T_\beta^{\mathrm{fr}, *}\otimes \omega_C)^*,
\]
where $ T_\beta^{\mathrm{fr}}$ denotes the free part of $T_{\beta}$. Now, for each $\alpha=1, \dots, r$ we have that $(K_\alpha, T_\alpha)$ is either $(0, L_\alpha)$ or $ (L_\alpha\otimes \OO_C(-Z_\alpha), \OO_{Z_\alpha})$, for a suitable closed subscheme $Z_\alpha\subset C$ (see \eqref{eqn: cases for quotient}). Denote by $J$ the collection of indices $\alpha$ such that  $(K_\alpha, T_\alpha)=(0, L_\alpha)$.  Therefore we have
\begin{align*}
    \bigoplus_{\alpha, \beta=1}^r \HH^0(C, K_\alpha \otimes  T_\beta^{\mathrm{fr}, *}\otimes \omega_C)&= \bigoplus_{\substack{\alpha\notin J\\ \beta\in J}}\HH^0(C, L_\alpha \otimes  L_\beta^{-1}\otimes \OO_C(-Z_\alpha) \otimes \omega_C)\\
    &\hookrightarrow \bigoplus_{\substack{\alpha\notin J\\ \beta\in J}}\HH^0(C, L_\alpha \otimes  L_\beta^{-1} \otimes \omega_C)=0.
\end{align*}
The vanishing of the obstruction space implies that all torus fixed points are smooth and unobstructed, which implies that all points are smooth and unobstructed as well\footnote{This follows by the fact that, in our case, the dimension of the obstruction space is upper semi-continuous, and that the closure of each torus orbit contains a torus fixed point.}.

Assume now by contradiction  that $g\geqslant 2$. Then \eqref{eqn: condition smooth l=1} implies that for all $\alpha\neq \beta$ we have
   \begin{align*}
       &\HH^1(L_\alpha\otimes L_\beta^{-1})=\HH^0(L_\alpha^{-1}\otimes L_\beta\otimes \omega_C)^*=0,\\
       &\HH^0(L_\alpha\otimes L_\beta^{-1})\hookrightarrow \HH^0(L_\alpha\otimes L_\beta^{-1}\otimes \omega_C)=0,
   \end{align*}
   which by Riemann--Roch  implies that 
   \[
   \deg L_\alpha -\deg L_\beta= g-1,
   \]
which is satisfied for all $\alpha\neq \beta$ only if $g=1$.
\end{proof}
We give some nontrivial instances of the vanishing \eqref{eqn: condition smooth l=1}.
\begin{example}\label{ex:smoothness}
  Let $C\cong \BP^1$, then all locally free sheaves admit a splitting $E=\bigoplus_{\alpha=1}^r L_\alpha$, and    \eqref{eqn: condition smooth l=1} translates into
\[
\deg L_\alpha -\deg L_{\beta} \leqslant 1,
\]
for all $\alpha\neq \beta$, i.e.~$E=L\otimes \left(\OO^{\oplus n}\oplus \OO(1)^{\oplus m}\right)$ where $L$ is any line bundle.

   Let $C$ be an elliptic curve and $E=\bigoplus_{\alpha=1}^r\OO_C(p_\alpha)$, where $p_\alpha\in C$ are distinct points. Then for each $\alpha\neq \beta$ the line bundles $\OO_C(p_\alpha-p_\beta)$ have no global sections, which implies that $(C, E)$ satifies \eqref{eqn: condition smooth l=1} and the corresponding Quot scheme is smooth.
\end{example}

\subsection{Smoothness in the nested case}
In the nested case, we show that there are non-trivial examples of smooth and unobstructed hyperquot schemes at least in genus 0.
\begin{theorem}\label{prop: smooth nested flag}
Let $L$ be any line bundle on $\BP^1$ and   $E=L\otimes \left(\OO^{\oplus n}\oplus \OO(1)^{\oplus m}\right)$ for some $n, m$. Then the hyperquot scheme $\Quot_{\BP^1}(E, \ns, \nd) $ is smooth and unobstructed for all $\nd$ and for all $\ns$.
\end{theorem}
\begin{proof}
The proof is similar to the ones of Propositions \ref{prop: smooth zero dim}, \ref{prop: smooth l=1}. Without loss of generality, we just need to prove our claim for a torus fixed point
    \[
 z= \bigoplus_{\alpha=1}^{r}\left(K_{l,\alpha}\hookrightarrow \cdots \hookrightarrow K_{1, \alpha} \hookrightarrow L_\alpha\onto T_{l,\alpha}\twoheadrightarrow\dots \twoheadrightarrow T_{1,\alpha} \right) \,\in\,   \Quot_{\BP^1}(E, \ns, \nd)^\TT.
\]
Set $E=\bigoplus_\alpha L_\alpha$. Consider the exact sequences
\[
\begin{tikzcd}[row sep=tiny]
0 \arrow{r} & K_{i,\alpha} \arrow{r} & K_{i-1,\alpha} \arrow{r} & Q_{i,\alpha} \arrow{r} & 0 \\
0 \arrow{r} & Q_{i,\alpha} \arrow{r} & T_{i,\alpha} \arrow{r} & T_{i-1,\alpha} \arrow{r} & 0
\end{tikzcd}
\]
for every $i=1,\dots, l$ and $\alpha=1, \dots, r$. Each $Q_{i,\alpha} $ either satisfies  $Q_{i, \alpha}\cong L_\alpha$ or is $0$-dimensional. Since $K_{i, \alpha}$ is locally free, by Serre duality we obtain the vanishings
\begin{align*}\label{eqn: vanishings}
\Ext^p_C(K_{i, \alpha}, T_{j,\beta})=\Ext^p_C(K_{i+1, \alpha}, T_{j,\beta})=\Ext^p_C(K_{i, \alpha}, Q_{j,\beta})=0, \quad p>0,
\end{align*}
which implies that $\Delta_z$, the map introduced in \eqref{ses:tg-space} is surjective. Therefore the obstruction space satisfies
\[
\begin{tikzcd}
h^1(\Gamma_z[-1])\arrow[hook]{r} & \displaystyle\bigoplus_{i=1}^l \Ext^1_{\BP^1}(K_i, T_i)=0, 
\end{tikzcd}
\]
by which we conclude that the obstruction space $h^1(\Gamma_z[-1])$ vanishes as well.
\end{proof}
In general, there are examples of smooth yet obstructed hyperquot schemes.
\begin{example}
Set $E=\OO_C^{\oplus r}$ and $\nd = \mathbf{0}$. Then for all $\ns$,  the hyperquot scheme is isomorphic to the partial flag variety $\Quot_C(\OO_C^{\oplus r}, \ns, \mathbf{0})\cong \Flag(\ns,r)$, but by \Cref{thm: pot 2 nest} it is endowed with a perfect obstruction theory of virtual dimension $(1-g(C)) \cdot \dim \Flag(\ns,r)$, which implies that it is smooth but obstructed as soon as $g(C)\geqslant 1$.
\end{example}

\section{Torus fixed locus}
\subsection{Torus action}\label{sec: torus action}
We fix $\ns=(s_1\leqslant \dots \leqslant s_l)$ an $l$-tuple of non-decreasing non-negative integers and $\nd=(d_1, \dots, d_l)$  an $l$-tuple of integers. Let $E$ be a locally free sheaf of rank $r>0$ on a smooth projective curve  $C$. Assume that $E=\bigoplus_{\alpha=1}^r L_\alpha$ splits into a sum of line bundles on $C$. Then $\Quot_C(E,\ns, \nd)$ admits the action of the algebraic torus $\TT=\BG_m^r$ as in \cite{Bifet}. Indeed, $\TT$ acts diagonally on the product $\prod_{i=1}^l \Quot_C(E,s_i, d_i)$ and the closed subscheme $\Quot_C(E,\ns, \nd)$ is $\TT$-invariant. Its fixed locus is determined by a straightforward generalisation of the main result of \cite{Bifet}, whose proof is analogous to \cite[Prop. 3.1]{MR_nested_Quot} (see also \cite[Sec.  3.1]{Chen_hyperquot}).

\begin{lemma}\label{lemma: fixed locus}
     If $E=\bigoplus_{\alpha=1}^r L_\alpha$, there is a scheme-theoretic identity
     \begin{align*}
    \Quot_C\left(E, \ns, \nd\right)^{\TT} = \coprod_{\ns=\ns_{1}+\dots + \ns_r}\coprod_{\nd=\nd_1+\dots +\nd_r}\prod_{\alpha=1}^r   \Quot_C(L_\alpha,\ns_\alpha, \nd_\alpha).
\end{align*}
\end{lemma}
In other words, Lemma \ref{lemma: fixed locus} says that all  torus fixed points
\begin{align*}
[K_{l}\hookrightarrow \cdots \hookrightarrow K_{1 } \hookrightarrow E \onto T_{l}\twoheadrightarrow\dots \twoheadrightarrow T_{1}]  \,\in\, \Quot_C\left(E, \ns, \nd\right)^{\TT}
\end{align*}
are  of the form
\begin{align}\label{eqn: fixed}
\bigoplus_{\alpha=1}^{r}\left(K_{l,\alpha}\hookrightarrow \cdots \hookrightarrow K_{1, \alpha} \hookrightarrow L_\alpha\onto T_{l,\alpha}\twoheadrightarrow\dots \twoheadrightarrow T_{1,\alpha} \right)\,\in\, \prod_{\alpha=1}^r\Quot_C\left(L_\alpha, \ns_{\alpha}, \nd_{\alpha} \right),
\end{align}
for some suitable $\ns_{\alpha}, \nd_{\alpha}$ for $\alpha=1, \dots, r$. Notice that for each pair $(K_{i,\alpha}, T_{i, \alpha})$ there are two possibilities. Namely, we have
\begin{align}\label{eqn: cases for quotient}
    \begin{cases}
        K_{i, \alpha} = L_\alpha\otimes\OO_C(-Z_{i,\alpha}) \\
        T_{i, \alpha}= \OO_{Z_{i, \alpha}}
    \end{cases}
    \mbox{ or } \quad 
     \begin{cases}
        K_{i, \alpha}=0 \\
        T_{i, \alpha}= L_\alpha
        \end{cases}
\end{align}
for some 0-dimensional closed subschemes $Z_{i,\alpha}\subset C $, satisfying the required nesting conditions. In particular, this implies that not all possible choices of splittings $\ns=\sum_{\alpha=1}^r\ns_{\alpha}$ and $\nd=\sum_{\alpha=1}^r\nd_{\alpha} $ contribute to the connected components of $ \Quot_C\left(E, \ns, \nd\right)^{\TT}$. In the next section we offer a convenient combinatorial description of the actual connected components appearing in  $ \Quot_C\left(E, \ns, \nd\right)^{\TT}$.

\subsection{Combinatorial description of   the fixed locus}\label{sec: comb fixed}
We keep the notation as in Section \ref{sec: torus action}. 
We give a combinatorial characterisation of the connected components appearing in the fixed locus of Lemma \ref{lemma: fixed locus}, following a description inspired by \cite[Sec. 3.1]{Chen_hyperquot}.
Given the   non-decreasing tuple $\ns=(s_1\leqslant \dots \leqslant s_l)$, define a new tuple $\nr=(r_1\geqslant \dots \geqslant r_l)$ by $ r_i=r-s_i$ for $i=1, \dots, l$. For convenience, we also set $r_0=r $ and $ r_{l+1}=0$.

\subsubsection{The unnested case} We start with the easier case of $l=1$, i.e.~no nesting conditions. By Lemma \ref{lemma: fixed locus} and \eqref{eqn: cases for quotient},  a fixed point as in  \eqref{eqn: fixed} decomposes as
\[
\begin{tikzcd}
\displaystyle\bigoplus_{\alpha=1}^{r_1}L_{c_{1,\alpha}}(-Z_{1, \alpha})
\arrow[hook]{r} & 
\displaystyle\bigoplus_{\alpha=1}^r L_\alpha\arrow[two heads]{r} & 
\displaystyle\bigoplus_{\alpha=1}^{r_1}\OO_{Z_{1, \alpha}}\oplus\bigoplus_{\substack{\beta\neq c_{1, \alpha}\\ \alpha=1, \dots, r_1}} L_\beta
\\
K_1 \arrow[equal]{u} & E\arrow[equal]{u} & T_1\arrow[equal]{u}
\end{tikzcd}
\]
for some closed 0-dimensional subschemes $Z_{1,\alpha}\subset C $ and some tuple of integers
\begin{align} \label{eqn: tuple ci}
   1 \leqslant c_{1, 1}<\dots < c_{1, r_1} \leqslant r, 
\end{align}
such that the inclusion of $K_1 \hookrightarrow E$ is ``component-wise'', i.e.~each summand $\OO_{Z_1, \alpha}$ is a quotient of $L_{c_{1, \alpha}}$. In other words, the tuple \eqref{eqn: tuple ci} selects exactly $r_1$ nonzero components (out  of the $r$ possible ones) of $K_1$.

Clearly, all tuples of the form \eqref{eqn: tuple ci} are in bijection with elements\footnote{By convention, we represent elements of $S_r$ as bijections $\sigma\colon \set{1, \dots, r}\to \set{1, \dots, r}.$} $\sigma\in S_r $, i.e. permutations of $r$ elements, with the condition that
\begin{align*}
    \sigma(\alpha)<\sigma(\beta) \quad\mbox{ if } \quad 1\leqslant \alpha<\beta\leqslant r_1  \quad\mbox{ or } \quad  r_1+1\leqslant \alpha<\beta\leqslant r_0=r,
\end{align*}
by declaring $c_{1, \alpha}=\sigma(\alpha)$. We denote the collection of such permutations by $P_{r_1}$.

Therefore, the fixed locus decomposes as
\begin{align}\label{eqn: fixed locus hyperquot combin}
\begin{split}
     \Quot_C\left(E, s_1, d_1\right)^{\TT}&=\coprod_{\sigma\in P_{ r_1}}\coprod_{\substack{d_1=\sum_{\alpha=1}^{r_1}n_{1,\alpha}+\\\sum_{\alpha=r_1+1}^{r}\deg L_{\sigma(\alpha)}}}  \prod_{\alpha=1}^{r_1}\Quot_C\left(L_{\sigma(\alpha)}, 0, n_{1,\alpha}\right)\\
    &\cong \coprod_{\sigma\in P_{ r_1}}\coprod_{\substack{d_1=\sum_{\alpha=1}^{r_1}n_{1,\alpha}+\\\sum_{\alpha=r_1+1}^{r}\deg L_{\sigma(\alpha)}}} \prod_{\alpha=1}^{r_1} C^{(n_{1,\alpha})},
\end{split}
\end{align}
where in the last line we used the isomorphism
\[ \Quot_C\left(L_{\sigma(\alpha)}, 0, n_{1,\alpha}\right)\cong C^{(n_{1,\alpha})}. \]

\subsubsection{The nested case}\label{sec: l geq 1} If the lenght of the nestings is $l\geqslant 1$, the flag of the kernels of a fixed point as in \eqref{eqn: fixed} is of the form
\begin{equation}\label{eqn: flag of kernels l}
\begin{tikzcd}[column sep=small]
\displaystyle\bigoplus_{1\leqslant \alpha\leqslant r_l}L_{c_{l,\alpha}}(-Z_{l, \alpha}) \arrow[hook]{r}
& \displaystyle\bigoplus_{\substack{l-1\leqslant j\leqslant l \\ r_{j+1}+1\leqslant \alpha \leqslant r_j}}L_{c_{j,\alpha}}(-Z_{l-1, \alpha}) \arrow[hook]{r}
& \cdots \arrow[hook]{r}
& \displaystyle\bigoplus_{\substack{1\leqslant j\leqslant l \\ r_{j+1}+1\leqslant \alpha \leqslant r_j}}L_{c_{j,\alpha}}(-Z_{1, \alpha})\arrow[hook]{r}
& E \\
K_l\arrow[equal]{u} &K_{l-1}\arrow[equal]{u} & & K_1\arrow[equal]{u}
\end{tikzcd}
\end{equation}
for some closed 0-dimensional subschemes $Z_{j,\alpha}\subset C $ and some tuples of integers
\begin{align}\label{eqn: c_i more}
\begin{split}
 &  1 \leqslant c_{l, 1}<\dots < c_{l, r_l} \leqslant r, \\
 &  1 \leqslant c_{l-1, r_{l}+1}<\dots < c_{l-1, r_{l-1}} \leqslant r, \\
&\vdots\\
 &  1 \leqslant c_{1, r_{2}+1}<\dots < c_{1, r_1} \leqslant r,
\end{split}
\end{align}
where the collections of indices $\set{c_{j,r_{j+1}+1}, \dots, c_{j, r_j}}$ are pairwise disjoint, for all $j=1,\dots, l$. 

As for the case of no nestings, each inclusion   $K_{j+1} \hookrightarrow K_{j}$ is ``component-wise'', i.e. each  $\OO_{Z_j, \alpha}$ is a quotient of $L_{c_{j, \alpha}}$ (as a summand of $E$). In other words, the tuples in \eqref{eqn: c_i more} select which components of  $K_j$ are nonzero.

Again, as for the case $l=1$, all tuples of the form \eqref{eqn: c_i more} are in bijection with elements $\sigma\in S_r $, with the condition that 
\begin{align*}
    \sigma(\alpha)<\sigma(\beta) \quad\mbox{ if } \quad r_{j+1}+1\leqslant \alpha<\beta\leqslant r_j  \quad\mbox{ for } \quad j=0, \dots, l,
\end{align*}
by declaring $c_{j, \alpha}=\sigma(\alpha)$ whenever $r_{j+1}+1\leqslant \alpha\leqslant r_j$ for $j=1, \dots, l$. We denote the collection of such permutations by $P_{\nr}$. Therefore \eqref{eqn: flag of kernels l} can be rewritten more compactly as
\[
\begin{tikzcd}
\displaystyle\bigoplus_{\alpha=1}^{r_l}L_{\sigma(\alpha)}(-Z_{l, \alpha})\arrow[hook]{r} & \displaystyle\bigoplus_{\alpha=1}^{r_{l-1}}L_{\sigma(\alpha)}(-Z_{l-1, \alpha})\arrow[hook]{r} & \cdots\arrow[hook]{r} & \displaystyle\bigoplus_{\alpha=1}^{r_1}L_{\sigma(\alpha)}(-Z_{1, \alpha})\arrow[hook]{r} & E\\ 
K_l\arrow[equal]{u} & K_{l-1}\arrow[equal]{u} & & K_1\arrow[equal]{u} &
\end{tikzcd}
\]
and the fixed locus decomposes as
\begin{align}\label{eqn: fixed locus of nested Quot with sigma nn}
\begin{split}
    \Quot_C\left(E, \ns, \nd\right)^{\TT}&=\coprod_{\sigma\in P_{ \nr}}\coprod_{\substack{(\nn_\alpha)_\alpha, \\ d_j=\sum_{\alpha=1}^{r_j} n_{j, \alpha}+\\\sum_{\alpha=r_j+1}^{r}\deg L_{\sigma(\alpha)}} }    \prod_{\alpha=1}^{r_1}\Quot_C\left(L_{\sigma(\alpha)}, \mathbf{0}, \nn_{\alpha}\right)\\
    &\cong \coprod_{\sigma\in P_{ \nr}}\coprod_{\substack{(\nn_\alpha)_\alpha, \\ d_j=\sum_{\alpha=1}^{r_j} n_{j, \alpha}+\\\sum_{\alpha=r_j+1}^{r}\deg L_{\sigma(\alpha)}} }\prod_{\alpha=1}^{r_1} C^{[\nn_\alpha]}.
\end{split}
\end{align}
Here, the second disjoint union is over all collection of tuples $(\nn_\alpha)_{\alpha=1, \dots, r_1}$ where for each $r_{j+1}+1\leqslant \alpha \leqslant r_{j}$, we have  $\nn_\alpha=(n_{1,\alpha}\leqslant \dots \leqslant n_{j,\alpha})$, and in the last equality we used that
\[
\Quot_C\left(L_{\sigma(\alpha)}, \mathbf{0}, \nn_{\alpha}\right)\cong C^{[\nn_\alpha]}.
\]
Notice also that, if $r_{j+1}+1\leq\alpha\leqslant r_j$, then 
\begin{align*}
    C^{[\nn_\alpha]}\cong \prod_{i=1}^{j} C^{(n_{i,\alpha}-n_{i-1, \alpha})},
\end{align*}
where we set $n_{0,\alpha}=0$ for convenience.

\subsection{Topological Euler characteristic}\label{sec:euler-top}
Consider $\bfk=\BC$ for this section. 
Let $\ns=(s_1 \leqslant \dots \leqslant s_l)$ be a non-decreasing tuple. Let $C$ be a  smooth projective curve of genus $g$ and $E=\bigoplus_{\alpha=1}^r L_\alpha$ a split locally free sheaf. We compute the generating series of topological Euler characteristics
\[
\mathsf{Z}^{\mathrm{top}}_{C,E, \ns}(\nq)=\sum_{\nd}\,e_{\mathrm{top}}(\Quot_C(E, \ns, \nd))\boldit{q}^{\nd}\in \BZ(\!( q_1, \dots, q_l)\!),
\]
without any smoothness assumption. 
\begin{theorem}\label{thm:euler-top}
There is an identity
\begin{align*}
\mathsf{Z}^{\mathrm{top}}_{C,E, \ns}(\nq)&=\left(\sum_{\sigma\in P_{ \nr}}\prod_{j=1}^{l} q_j^{ \sum_{\alpha=r_j+1}^{r}\deg L_{\sigma(\alpha)}}\right)\cdot\prod_{1\leqslant i\leqslant j\leqslant l}(1-q_i\cdots q_j)^{(2g-2)(r_j-r_{j+1})}.
    \end{align*}
\end{theorem}
\begin{proof}
The topological Euler characteristic of a space with a torus action coincides with the topological Euler characteristic of its torus fixed locus. Therefore by \eqref{eqn: fixed locus of nested Quot with sigma nn} we have
    \begin{align*}
          \mathsf{Z}^{\mathrm{top}}_{C,E, \ns}(\nq)&=\sum_{\nd}\nq^{\nd}\sum_{\sigma\in P_{ \nr}}\sum_{\substack{\nn=(\nn_\alpha)_\alpha, \\ d_j=\sum_{\alpha=1}^{r_j} n_{j, \alpha}+\\\sum_{\alpha=r_j+1}^{r}\deg L_{\sigma(\alpha)}} }\prod_{\alpha=1}^{r_1}e_{\mathrm{top}}(C^{[\nn_\alpha]})\\
          &=\sum_{\sigma\in P_{ \nr}}\prod_{j=1}^{l} q_j^{ \sum_{\alpha=r_j+1}^{r}\deg L_{\sigma(\alpha)}}\prod_{j=1}^{l} \prod_{\alpha=r_{j+1}+1}^{r_j}\sum_{\nn_\alpha}q_1^{n_{1, \alpha}} \cdots q_j^{n_{j,\alpha}} e_{\mathrm{top}}(C^{[\nn_\alpha]}).
    \end{align*}
By \cite[Cor. 4.5]{MR_nested_Quot}, we have
    \begin{align*}
\sum_{\nn_\alpha}q_1^{n_{1, \alpha}} \cdots q_j^{n_{j,\alpha}} e_{\mathrm{top}}(C^{[\nn_\alpha]})=\prod_{i=1}^j\,(1-q_i\cdots q_j)^{2g-2},
    \end{align*}
    by which we conclude that 
    \begin{align*}
         \mathsf{Z}^{\mathrm{top}}_{C,E, \ns}(\nq)&=\sum_{\sigma\in P_{ \nr}}\prod_{j=1}^{l} q_j^{ \sum_{\alpha=r_j+1}^{r}\deg L_{\sigma(\alpha)}}\prod_{1\leqslant i\leqslant j\leqslant l}(1-q_i\cdots q_j)^{(2g-2)(r_j-r_{j+1})}.\qedhere
    \end{align*}
\end{proof}
Setting $C=\BP^1$ and $E=\OO_{\BP^1}^{\oplus r}$ in the previous theorem one recovers Chen's result
\cite[Thm~3]{Chen_hyperquot}.

\section{Motives of hyperquot schemes}
\subsection{Grothendieck ring of varieties}
The \emph{Grothendieck ring of~~$\bfk$-varieties}, denoted $K_0(\Var_{\bfk})$, is the free abelian group generated by isomorphism classes $[X]$ of $\bfk$-varieties, modulo the scissor relations, namely the identities $[X] = [Z] + [X\setminus Z]$ whenever $Z \into X$ is a closed $\bfk$-subvariety of $X$. The ring structure is defined by fibre product. The ring identity is $[\Spec \bfk]$. The main rules for calculations in $K_0(\Var_{\bfk})$ are the following:
\begin{enumerate}
    \item If $X \to Y$ is a geometric bijection, i.e.~a bijective morphism, then $[X] = [Y]$.
    \item If $X \to Y$ is Zariski locally trivial with fibre $F$, then $[X] = [Y]\cdot [F]$.
\end{enumerate}

These are, indeed, the only properties that we will use. Recall that the \emph{Lefschetz motive} is defined to be the class $\BL = [\BA^1_{\bfk}] \in K_0(\Var_{\bfk})$.

\subsection{Białynicki-Birula decomposition for hyperquot schemes}
We recall the classical result of Białynicki-Birula.

\begin{theorem}[{Białynicki-Birula \cite[Sec.~4]{BB_some_theorems}}]\label{thm: BB}
Let $X$ be a smooth projective scheme with a $\BG_m$-action and let $\set{X_i}_i$ be the connected components of the $\BG_m$-fixed locus $X^{\BG_m} \subset X$. Then there exists a locally closed stratification $X=\coprod_i X^+_i$, such that each $X^+_i\to X_i$ is an affine fibre bundle. Moreover, for every closed point $x\in X_i$, the tangent space is given by $T_x(X^+_i)=T_x(X)^{\fix}\oplus T_x(X)^{+}$, where $T_x(X)^{\fix}$ (resp. $T_x(X)^{+} $) denotes the $\BG_m$-fixed  (resp. positive) part of $T_x(X)$. In particular, the relative dimension of $X^+_i\to X_i$ is equal to $\dim T_x(X)^{+}$ for $x \in X_i$.
\end{theorem}

We now determine a  Białynicki-Birula decomposition for $\Quot_C(E, \ns, \nd)$ whenever it is smooth and unobstructed.

Let $\BG_m\hookrightarrow \TT$ be the generic 1-parameter subtorus given by $w\mapsto (w, w^2,\dots, w^r)$; it is clear that $\Quot_C(E, \ns, \nd)^\TT=\Quot_C(E, \ns, \nd)^{\BG_m}$. Let
\[
Q_{\sigma, \nn}=\prod_{\alpha=1}^{r_1} C^{[\nn_\alpha]}\subset  \Quot_C(E,\ns, \nd)^{\BG_m}
\]
be the connected component of the fixed locus corresponding to the data  $\sigma, \nn=(\nn_\alpha)_\alpha$ as in \eqref{eqn: fixed locus hyperquot combin}. 

Recall that we set $\nr=(r_1\geqslant \dots \geqslant r_l)$, where $r_i=r-s_i$ and we set $r_0=r$ and $r_{l+1}=0$ for convenience. For a fixed permutation $\sigma\in P_{\nr}$, define for all $\alpha, \beta=1, \dots, r$ 
\[
\epsilon^{\sigma}_{\alpha, \beta}=\begin{cases}
    1 & \sigma(\beta)> \sigma(\alpha),\\
    0 & \sigma(\beta)\leqslant  \sigma(\alpha).
\end{cases}
\]
Define moreover, for all $i=1, \dots, l+1$ and $\alpha=1, \dots, r$
\begin{equation}\label{eqn:h_0,h_2}
\begin{split}
  h^\sigma_0(i, \alpha)&=\sum_{\beta=r_{i}+1}^{r_{i-1}} \epsilon^\sigma_{\alpha, \beta}\\
   h^\sigma_1(i, \alpha)&=  \sum_{\beta=r_{i+1}+1}^{r_i} \epsilon^\sigma_{\beta, \alpha} + \sum_{\beta=r_i+1}^{r_{i-1}}\epsilon^\sigma_{\alpha,\beta},\\
  h^\sigma_2(i, g, v_1, \dots, v_r)&= \sum_{\alpha=r_{i+1}+1}^{r_{i}}\sum_{\beta=r_i+1}^{r} ( v_{\sigma(\beta)}-v_{\sigma(\alpha)}+ 1-g )\cdot \epsilon^\sigma_{\alpha,\beta},
\end{split}
\end{equation}
where $v_1,\ldots,v_r$ are integers. Notice that  $h^{\sigma}_0, h^\sigma_1, h^\sigma_2$ are algorithmically determined simply by the initial data of $\nr, \sigma$.
\begin{prop}\label{prop: BB for Quot}
Let $C$ be a smooth projective curve and $E=\bigoplus_{\alpha=1}^rL_\alpha$ such that the hyperquot scheme $\Quot_C(E, \ns, \nd)$ is smooth and unobstructed. Then there is a locally closed stratification 
\[
\Quot_C(E, \ns, \nd)=\coprod_{\sigma\in P_{ \nr}}\coprod_{\substack{\nn=(\nn_\alpha)_\alpha, \\ d_j=\sum_{\alpha=1}^{r_j} n_{j, \alpha}+\\\sum_{\alpha=r_j+1}^{r}\deg L_{\sigma(\alpha)}} }Q^{+}_{{\sigma, \nn}},
\]
where the second disjoint union is over all collections of tuples $(\nn_\alpha)_{\alpha=1, \dots, r_1}$ such that, for each $r_{j+1}+1\leqslant \alpha \leqslant r_{j}$, we have  $\nn_\alpha=(n_{1,\alpha}\leqslant \dots \leqslant n_{j,\alpha})$,
 and $Q^{+}_{{\sigma, \nn}}\to Q_{{\sigma, \nn}}$ is an affine fibre bundle of relative dimension 
 \begin{align*}
\sum_{i=1}^l\sum_{\alpha=1}^{r_i}n_{i, \alpha}\cdot h^\sigma_1(i, \alpha)+ \sum_{i=1}^l h^\sigma_{2}(i, g, \deg L_1,\dots,  \deg L_r).
\end{align*}
\end{prop}
\begin{proof}
By our assumption the hyperquot scheme $ \Quot_C(E, \ns, \nd)$ is smooth (and projective), therefore we can apply Theorem \ref{thm: BB}. Fix $\sigma, \nn$. To compute the relative dimension of the affine fibre bundle $Q^{+}_{{\sigma, \nn}}\to Q_{{\sigma, \nn}}$, we consider a $\TT$-fixed point $z\in Q_{{\sigma, \nn}}$ and compute the positive part of the tangent space at $z$. Say that the flag
\[
\begin{tikzcd}
\displaystyle\bigoplus_{\alpha=1}^{r_l}L_{\sigma(\alpha)}(-Z_{l, \alpha}) \arrow[hook]{r} & \displaystyle\bigoplus_{\alpha=1}^{r_{l-1}}L_{\sigma(\alpha)}(-Z_{l-1, \alpha})\arrow[hook]{r} & \cdots \arrow[hook]{r} & \displaystyle\bigoplus_{\alpha=1}^{r_1}L_{\sigma(\alpha)}(-Z_{1, \alpha}) \arrow[hook]{r} & E
\\
K_l\arrow[equal]{u} & K_{l-1}\arrow[equal]{u} & & K_1\arrow[equal]{u} &
\end{tikzcd}
\]
corresponds to the point $z$.
Let $\BG_m\hookrightarrow \TT$ be the generic 1-parameter subtorus given by $w\mapsto (w, w^2,\dots, w^r)$. Since by our assumption the hyperquot scheme is unobstructed, the class in $K$-theory of the tangent bundle at $z$ coincides with the class of the dual $\BE_{\QQ}^\vee$ of the perfect obstruction theory, which by \Cref{rmk:same-K} can be computed as $\Cone(\RR_\pi\Delta)|_z$. In particular, the tangent space decomposes $\BG_m$-equivariantly as
 \begin{align*}
    T_z=\sum_{\alpha, \beta=1}^r\left(\sum_{i=1}^l \RR\Hom (K_{i, \alpha},T_{i, \beta})-\sum_{i=1}^{l-1} \RR\Hom (K_{i+1, \alpha},T_{i, \beta})\right) w^{\beta-\alpha} \in K^0_{\BG_m}(z),
\end{align*}
where, to set notation, we decompose the kernels and the quotients as
\begin{align*}
    K_i=\bigoplus_{\alpha=1}^r K_{i, \alpha}, \quad T_i=\bigoplus_{\alpha=1}^r T_{i, \alpha}, \quad i=1, \dots, l,
\end{align*}
and we use the same notation as in \Cref{sec: comb fixed}.
Therefore, the relative dimension of the affine fibre bundle coincides with
\begin{align}\label{eqn: loc contrub}
\begin{split}
\rk T_z^{ +}&=\sum_{1\leq\alpha<\beta\leqslant r}\left(\sum_{i=1}^l \rk\RR\Hom (K_{i, \alpha},T_{i, \beta})-\sum_{i=1}^{l-1} \rk\RR\Hom (K_{i+1, \alpha},T_{i, \beta})\right)\\
&=\sum_{\alpha,\beta=1}^r\left(\sum_{i=1}^l \rk\RR\Hom (K_{i, \sigma(\alpha)},T_{i, \sigma(\beta)})-\sum_{i=1}^{l-1} \rk\RR\Hom (K_{i+1, \sigma(\alpha)},T_{i, \sigma(\beta)})\right)\cdot \epsilon^{\sigma}_{\alpha,\beta}\\
&=\sum_{i=1}^l\sum_{\alpha,\beta=1}^r\left( \rk\RR\Hom (K_{i, \sigma(\alpha)},T_{i, \sigma(\beta)})- \rk\RR\Hom (K_{i+1, \sigma(\alpha)},T_{i, \sigma(\beta)})\right)\cdot \epsilon^{\sigma}_{\alpha,\beta}
\end{split}
\end{align}
where we set for convenience $K_{l+1, \sigma(\alpha)}=0$.
We consider now all the possible values of $\rk\RR\Hom (K_{i, \sigma(\alpha)},T_{j, \sigma(\beta)}) $.

\noindent
\underline{Case I.} If $\alpha> r_i$, then $K_{i, \sigma(\alpha)}=0$ and
\[
\rk\RR\Hom (K_{i, \sigma(\alpha)},T_{j, \sigma(\beta)}) =0.
\]

\noindent
\underline{Case II.} If $\alpha \leqslant r_i $ and $ \beta\leqslant r_j$, we have
\begin{align*}
    \rk\RR\Hom (K_{i, \sigma(\alpha)},T_{j, \sigma(\beta)})&= \rk\RR\Hom (L_{\sigma(\alpha)}(-Z_{i, \alpha}),\OO_{Z_{j,\beta}})\\
    &= n_{j,\beta}.
\end{align*}

\noindent
\underline{Case III.} If $\alpha \leqslant r_i $ and $ \beta> r_j$, by Riemann-Roch we have
\begin{align*}
    \rk\RR\Hom (K_{i, \sigma(\alpha)},T_{j, \sigma(\beta)})&=\rk\RR\Hom (L_{\sigma(\alpha)}(-Z_{i, \alpha}),L_{\sigma(\beta)})\\
    &=\deg L_{\sigma(\beta)}-\deg L_{\sigma(\alpha)}+n_{i,\alpha}+ 1-g.
\end{align*}
We compute now all the contributions to the rank in \eqref{eqn: loc contrub}. Fix $i=1, \dots, l$. 

\noindent
\underline{Case I.}  If $\alpha> r_i$, then $K_{i, \sigma(\alpha)}=K_{i+1, \sigma(\alpha)}=0 $ and therefore 
\[
\left( \rk\RR\Hom (K_{i, \sigma(\alpha)},T_{i, \sigma(\beta)})- \rk\RR\Hom (K_{i+1, \sigma(\alpha)},T_{i, \sigma(\beta)})\right)\cdot \epsilon^{\sigma}_{\alpha,\beta}=0.
\]

\noindent
\underline{Case II.} If $\alpha \leqslant r_{i+1}$ and $ \beta\leqslant r_i$, we have
\begin{align*}
    \left( \rk\RR\Hom (K_{i, \sigma(\alpha)},T_{i, \sigma(\beta)})- \rk\RR\Hom (K_{i+1, \sigma(\alpha)},T_{i, \sigma(\beta)})\right)\cdot \epsilon^{\sigma}_{\alpha,\beta} 
    &=(n_{i,\beta}-n_{i,\beta})\cdot \epsilon^{\sigma}_{\alpha,\beta}\\
    &=0.
\end{align*}

\noindent
\underline{Case III.} If $r_{i+1}+1\leqslant \alpha \leqslant r_i $ and $ \beta\leqslant r_i$, we have
\begin{align*}
    \left( \rk\RR\Hom (K_{i, \sigma(\alpha)},T_{i, \sigma(\beta)})- \rk\RR\Hom (K_{i+1, \sigma(\alpha)},T_{i, \sigma(\beta)})\right)\cdot \epsilon^{\sigma}_{\alpha,\beta}=n_{i, \beta}\cdot \epsilon^{\sigma}_{\alpha,\beta}.
\end{align*}

\noindent
\underline{Case IV.} If $\alpha \leqslant r_{i+1} $ and $ \beta> r_i$, we have
\begin{align*}
    \left( \rk\RR\Hom (K_{i, \sigma(\alpha)},T_{i, \sigma(\beta)})- \rk\RR\Hom (K_{i+1, \sigma(\alpha)},T_{i, \sigma(\beta)})\right)\cdot \epsilon^{\sigma}_{\alpha,\beta}=(n_{i, \alpha}-n_{i+1, \alpha})\cdot \epsilon^{\sigma}_{\alpha,\beta}.
\end{align*}

\noindent
\underline{Case V.} If $r_{i+1}+1\leqslant \alpha \leqslant r_i $ and $ \beta> r_i$, we have
\begin{multline*}
    \left( \rk\RR\Hom (K_{i, \sigma(\alpha)},T_{i, \sigma(\beta)})- \rk\RR\Hom (K_{i+1, \sigma(\alpha)},T_{i, \sigma(\beta)})\right)\cdot \epsilon^{\sigma}_{\alpha,\beta}\\
    =( \deg L_{\sigma(\beta)}-\deg L_{\sigma(\alpha)}+n_{i,\alpha}+ 1-g )\cdot \epsilon^{\sigma}_{\alpha,\beta}.
\end{multline*}
Patching everything together, \eqref{eqn: loc contrub} equals
\begin{multline}\label{eqn: tvir conto dim}
    \rk T_z^{ +}=\sum_{i=1}^l\left(\sum_{\alpha=r_{i+1}+1}^{r_i}\sum_{\beta=1}^{r_i}n_{i, \beta}\cdot \epsilon^{\sigma}_{\alpha, \beta}+\sum_{\alpha=1}^{r_{i+1}}\sum_{\beta=r_i+1}^{r}(n_{i, \alpha}-n_{i+1, \alpha})\cdot \epsilon^{\sigma}_{\alpha,\beta}\right.\\ \left. +\sum_{\alpha=r_{i+1}+1}^{r_{i}}\sum_{\beta=r_i+1}^{r} ( \deg L_{\sigma(\beta)}-\deg L_{\sigma(\alpha)}+n_{i,\alpha}+ 1-g )\cdot \epsilon^{\sigma}_{\alpha,\beta}  \right).
\end{multline}
We simplify the above expression. We have that
\begin{equation}\label{eqn: trick sulle differenze}
\begin{split}
    \sum_{i=1}^l 
    &\left( \sum_{\alpha=1}^{r_{i+1}}\sum_{\beta=r_i+1}^{r}(n_{i, \alpha}-n_{i+1, \alpha})\cdot \epsilon^{\sigma}_{\alpha,\beta}+ \sum_{\alpha=r_{i+1}+1}^{r_{i}}\sum_{\beta=r_i+1}^{r}n_{i,\alpha}\cdot \epsilon^{\sigma}_{\alpha,\beta}\right)\\
    &=  \sum_{i=1}^l \left(  \sum_{\alpha=1}^{r_{i}}\sum_{\beta=r_i+1}^{r}n_{i, \alpha}\cdot \epsilon^{\sigma}_{\alpha,\beta} -  \sum_{\alpha=1}^{r_{i+1}}\sum_{\beta=r_i+1}^{r} n_{i+1, \alpha}\cdot \epsilon^{\sigma}_{\alpha,\beta} \right)\\
    &= \sum_{i=1}^l\left( \sum_{\alpha=1}^{r_i}\sum_{\beta=r_i+1}^{r_{i-1}}n_{i, \alpha}\cdot \epsilon^{\sigma}_{\alpha,\beta} +\sum_{\alpha=1}^{r_{i}}\sum_{\beta=r_{i-1}+1}^{r}n_{i, \alpha}\cdot \epsilon^{\sigma}_{\alpha,\beta}  -  \sum_{\alpha=1}^{r_{i+1}}\sum_{\beta=r_i+1}^{r} n_{i+1, \alpha}\cdot \epsilon^{\sigma}_{\alpha,\beta} \right)\\
    &= \sum_{i=1}^l\left( \sum_{\alpha=1}^{r_i}\sum_{\beta=r_i+1}^{r_{i-1}}n_{i, \alpha}\cdot \epsilon^{\sigma}_{\alpha,\beta} \right).
\end{split}
\end{equation}
Therefore by plugging \eqref{eqn: trick sulle differenze} into \eqref{eqn: tvir conto dim} and suitably relabelling $\alpha$ and $\beta$ in the first summand, we obtain
\begin{multline*}
     \rk T_z^{ +}=\sum_{i=1}^l\left(\sum_{\alpha=1}^{r_i}n_{i, \alpha}\cdot \left( \sum_{\beta=r_{i+1}+1}^{r_i} \epsilon^{\sigma}_{\beta, \alpha} + \sum_{\beta=r_i+1}^{r_{i-1}}\epsilon^{\sigma}_{\alpha,\beta}\right) \right.\\
     \left.+ \sum_{\alpha=r_{i+1}+1}^{r_{i}}\sum_{\beta=r_i+1}^{r} ( \deg L_{\sigma(\beta)}-\deg L_{\sigma(\alpha)}+ 1-g )\cdot \epsilon^{\sigma}_{\alpha,\beta}  \vphantom{}\right),
\end{multline*}
which concludes the argument.
\end{proof}

Specialising in Proposition \ref{prop: BB for Quot} the rank of all the quotients to be $0$, i.e.~$r_1=\dots =r_l=r$ and $r_{l+1}=0$, we recover the decomposition computed in  \cite[Prop. 3.4]{MR_nested_Quot}.

\subsection{Motivic partition function}
Let $C$ be a smooth projective curve, $E$ a locally free sheaf of rank $r$ and fix $\ns=(s_1\leqslant \dots \leqslant s_l)$ . Define
\[
\mathsf{Z}_{C,E, \ns}(\nq)=\sum_{\nd}\,\bigl[\Quot_C(E, \ns, \nd)\bigr]\boldit{q}^{\nd}\in K_0(\Var_{\bfk})(\!( q_1, \dots, q_l)\!),
\]
where $\nd=(d_1, \dots, d_l)$ and we use the multivariable notation $\nq = (q_1,\ldots,q_l)$ and $\nq^{\nd}=\prod_{i=1}^l q^{d_i}$. If $l=1$, $E=L$ is a line bundle and $\ns=0$, then $\mathsf{Z}_{C,L,0}(q)$ is simply Kapranov's motivic zeta function \cite{Kapranov_rational_zeta}
\begin{equation}\label{eqn:kapranov}
\mathsf{Z}_{C,L,0}(q)=\zeta_C(q)=\sum_{d\geqslant 0}\,\bigl[C^{(d)}\bigr]q^d.
\end{equation}
We prove an  expression for the motive of the hyperquot scheme, in terms of Kapranov motivic zeta functions.
\begin{theorem}\label{thm: theorem motivic all g}
Let $C$ be a smooth projective curve of genus $g$ and $E=\bigoplus_{\alpha=1}^rL_\alpha$ such that the hyperquot scheme $ \Quot_C(E, \ns, \nd)$ is smooth and unobstructed. There is an identity
\begin{multline*}
\mathsf{Z}_{C,E, \ns}(\nq)=\sum_{\sigma\in P_{ \nr}}\prod_{j=1}^{l}\left(\BL^{h^\sigma_{2}(j,g, \deg L_1,\dots,  \deg L_r)} q_j^{ \sum_{\alpha=r_j+1}^{r}\deg L_{\sigma(\alpha)}}\right) \\
\cdot \prod_{1\leqslant i\leqslant j\leqslant l}\prod_{\alpha=r_{j+1}+1}^{r_j}\zeta_C\left(\BL^{ h^\sigma_0(i,\alpha)+r_i-r_j+\alpha-r_{j+1}-1}q_i\cdots q_j \right).
\end{multline*}
In particular, $\mathsf{Z}_{C,E, \ns}(\nq)$ is a rational function.
\end{theorem}

\begin{proof}
By the assumptions,  to compute the motive of  $\Quot_C(E, \ns, \nd)$ we can exploit the stratification of Proposition \ref{prop: BB for Quot}.   Every stratum is a Zariski locally trivial fibration  over a connected component of the fixed locus, with fibre an affine space whose dimension we computed in \Cref{prop: BB for Quot}. Therefore we have
\begin{align*}
     \mathsf{Z}_{C,E,\ns}(\nq)&=
     \sum_{\nd}\nq^{\nd}\sum_{\sigma\in P_{ \nr}}\sum_{\substack{\nn=(\nn_\alpha)_\alpha, \\ d_j=\sum_{\alpha=1}^{r_j} n_{j, \alpha}+\\ \sum_{\alpha=r_j+1}^{r}\deg L_{\sigma(\alpha)}} } \BL^{ \sum_{j=1}^l\sum_{\alpha=1}^{r_j}n_{j, \alpha}\cdot h^\sigma_1(j, \alpha)+ \sum_{j=1}^l h^\sigma_{2}(j, g, \deg L_1,\dots,  \deg L_r)}\cdot \prod_{\alpha=1}^{r_1}\left[C^{[\nn_\alpha]}\right]\\
     &=\sum_{\nd}\sum_{\sigma\in P_{ \nr}}\prod_{j=1}^{l}\left(\BL^{h^\sigma_{2}(j,g, \deg L_1,\dots,  \deg L_r)} q_j^{ \sum_{\alpha=r_j+1}^{r}\deg L_{\sigma(\alpha)}}\right)\\
     &\qquad\qquad\qquad \cdot\sum_{\substack{\nn=(\nn_\alpha)_\alpha, \\ d_j=\sum_{\alpha=1}^{r_j} n_{j, \alpha}+\\\sum_{\alpha=r_j+1}^{r}\deg L_{\sigma(\alpha)}} }\prod_{j=1}^{l}\left(\BL^{\sum_{\alpha=1}^{r_j}n_{j, \alpha}\cdot h^\sigma_1(j, \alpha)}q_j^{\sum_{\alpha=1}^{r_j} n_{j, \alpha}}\right) \prod_{\alpha=1}^{r_1}\left[C^{[\nn_\alpha]}\right]\\
     &=\sum_{\sigma\in P_{ \nr}}\prod_{j=1}^{l}\left(\BL^{h^\sigma_{2}(j,g, \deg L_1,\dots,  \deg L_r)} q_j^{ \sum_{\alpha=r_j+1}^{r}\deg L_{\sigma(\alpha)}}\right)\\
     &\qquad\qquad\qquad\cdot\sum_{\substack{\nn=(\nn_\alpha)_\alpha} }\prod_{j=1}^{l}\left(\BL^{\sum_{\alpha=1}^{r_j}n_{j, \alpha}\cdot h^\sigma_1(j, \alpha)}q_j^{\sum_{\alpha=1}^{r_j} n_{j, \alpha}}\right) \prod_{\alpha=1}^{r_1}\left[C^{[\nn_\alpha]}\right].
\end{align*}
We redistribute the variables to simplify the above expression. We have
\[
\prod_{j=1}^{l}q_j^{\sum_{\alpha=1}^{r_j} n_{j, \alpha}}=\prod_{j=1}^{l}\prod_{\alpha=r_{j+1}+1}^{r_j}q_1^{n_{1,\alpha}}\cdots q_j^{n_{j,\alpha}}.
\]
Similarly, we have
\[
\prod_{j=1}^{l}\BL^{\sum_{\alpha=1}^{r_j}n_{j, \alpha}\cdot h^\sigma_1(j, \alpha)}=\prod_{j=1}^{l}\prod_{\alpha=r_{j+1}+1}^{r_j} \BL^{n_{1, \alpha}\cdot h^\sigma_1(1, \alpha)}\cdots  \BL^{n_{j, \alpha}\cdot h^\sigma_1(j, \alpha)}
\]
and 
\[
\prod_{\alpha=1}^{r_1}\left[C^{[\nn_\alpha]}\right]=\prod_{j=1}^l \prod_{\alpha=r_{j+1}+1}^{r_j}\left[C^{[\nn_\alpha]}\right].
\]
Therefore we have
\begin{align*}
      \mathsf{Z}_{C,E,\ns}(\nq)
      &=
      \sum_{\sigma\in P_{ \nr}}\prod_{j=1}^{l}\left(\BL^{h^\sigma_{2}(j,g, \deg L_1,\dots,  \deg L_r)} q_j^{ \sum_{\alpha=r_j+1}^{r}\deg L_{\sigma(\alpha)}}\right)\\
      &\qquad\qquad\cdot \sum_{\substack{\nn=(\nn_\alpha)_\alpha} }\prod_{j=1}^{l} \prod_{\alpha=r_{j+1}+1}^{r_j}\left((\BL^{h^\sigma_1(1, \alpha)}q_1)^{n_{1, \alpha}} \cdots (\BL^{h^\sigma_1(j, \alpha)}q_j)^{n_{j,\alpha}} \left[C^{[\nn_\alpha]}\right] \right)\\
      &=\sum_{\sigma\in P_{ \nr}}\prod_{j=1}^{l}\left(\BL^{h^\sigma_{2}(j,g, \deg L_1,\dots,  \deg L_r)} q_j^{ \sum_{\alpha=r_j+1}^{r}\deg L_{\sigma(\alpha)}}\right)\\
      &\qquad\qquad\cdot\prod_{j=1}^{l} \prod_{\alpha=r_{j+1}+1}^{r_j}\sum_{\nn_\alpha}(\BL^{h^\sigma_1(1, \alpha)}q_1)^{n_{1, \alpha}} \cdots (\BL^{h^\sigma_1(j, \alpha)}q_j)^{n_{j,\alpha}} \left[C^{[\nn_\alpha]}\right].
\end{align*}
Notice now that the last sum is a generating series of motives of hyperquot schemes $\Quot_C(\OO_C, \mathbf{0}, \nn_\alpha)$, where the rank of all the quotients is $0$.
The associated  generating series was computed in \cite[Thm. 4.2]{MR_nested_Quot}\footnote{Alternatively one could directly see the factorisation exploiting the fact that the nested Hilbert scheme on a smooth projective curve is a product of symmetric products.} and equals
\begin{align*}
\sum_{\nn_\alpha}\,(\BL^{h^\sigma_1(1, \alpha)}q_1)^{n_{1, \alpha}} \cdots (\BL^{h^\sigma_1(j, \alpha)}q_j)^{n_{j,\alpha}} \left[C^{[\nn_\alpha]}\right]=\prod_{i=1}^j\zeta_C\left(\BL^{\sum_{k=i}^j h^\sigma_1(k, \alpha)}q_i\cdots q_j \right).    
\end{align*}
We conclude that 
\begin{align*}
     \mathsf{Z}_{C,E, \ns}(\nq)&=\sum_{\sigma\in P_{ \nr}}\prod_{j=1}^{l}\left(\BL^{h^\sigma_{2}(j,g, \deg L_1,\dots,  \deg L_r)} q_j^{ \sum_{\alpha=r_j+1}^{r}\deg L_{\sigma(\alpha)}}\right) \\
     &\qquad \qquad\qquad\cdot\prod_{j=1}^{l}\prod_{\alpha=r_{j+1}+1}^{r_j}\prod_{i=1}^j\zeta_C\left(\BL^{\sum_{k=i}^j h^\sigma_1(k, \alpha)}q_i\cdots q_j \right)\\
     &=\sum_{\sigma\in P_{ \nr}}\prod_{j=1}^{l}\left(\BL^{h^\sigma_{2}(j,g, \deg L_1,\dots,  \deg L_r)} q_j^{ \sum_{\alpha=r_j+1}^{r}\deg L_{\sigma(\alpha)}}\right)\\
     &\qquad\qquad \qquad\cdot\prod_{1\leqslant i\leqslant j\leqslant l}\prod_{\alpha=r_{j+1}+1}^{r_j}\zeta_C\left(\BL^{ h^\sigma_0(i,\alpha)+r_i-r_j+\alpha-r_{j+1}-1}q_i\cdots q_j \right),
\end{align*}
where in the last equality we used Lemma \ref{lemma: comb with h0}. Rationality of $   \mathsf{Z}_{C,E, \ns}(\nq)$ follows from the rationality of $\zeta_C$ proved in \cite{Kapranov_rational_zeta}.
\end{proof}

Once more, specialising in Theorem \ref{thm: theorem motivic all g} the ranks of all the quotients to be $0$, i.e.~setting $r_1=\dots =r_l=r$ and $r_{l+1}=0$, one recovers the product formula in \cite[Thm. 4.2]{MR_nested_Quot} which, in turn, generalises the corresponding computation in the unnested case \cite{BFP_motives,Ric_motive_quot_locally_free}. 

We prove now the combinatorial result we used in the proof of Theorem \ref{thm: theorem motivic all g}.

\begin{lemma}\label{lemma: comb with h0}
    Let $1\leqslant i\leqslant j\leqslant r$ and $r_{j+1}+1\leqslant \alpha\leqslant r_j$. We have
    \[
    \sum_{k=i}^j h^\sigma_1(k, \alpha)= h^\sigma_0(i,\alpha)+r_i-r_j+\alpha-r_{j+1}-1.
    \]
\end{lemma}
\begin{proof}
    By definition, we have
\begin{align*}
\sum_{k=i}^j h^\sigma_1(k, \alpha)&=  \sum_{k=i}^j \left(\sum_{\beta=r_{k+1}+1}^{r_k} \epsilon^{\sigma}_{\beta, \alpha} + \sum_{\beta=r_k+1}^{r_{k-1}}\epsilon^{\sigma}_{\alpha,\beta}\right)\\
&=\sum_{\beta=r_{i+1}+1}^{r_i} \epsilon^{\sigma}_{\beta, \alpha} + \sum_{\beta=r_i+1}^{r_{i-1}}\epsilon^{\sigma}_{\alpha,\beta}+\sum_{\beta=r_{i+2}+1}^{r_{i+1}} \epsilon^{\sigma}_{\beta, \alpha} \\
&\qquad\qquad\qquad+ \sum_{\beta=r_{i+1}+1}^{r_{i}}\epsilon^{\sigma}_{\alpha,\beta}+\dots +\sum_{\beta=r_{j+1}+1}^{r_j} \epsilon^{\sigma}_{\beta, \alpha} + \sum_{\beta=r_j+1}^{r_{j-1}}\epsilon^{\sigma}_{\alpha,\beta}\\
        &=\sum_{k=i}^{j-1}\sum_{\beta=r_{k+1}+1}^{r_{k}}(\epsilon^{\sigma}_{\alpha,\beta}+\epsilon^{\sigma}_{\beta,\alpha})+ \sum_{\beta=r_{i}+1}^{r_i-1} \epsilon^{\sigma}_{\alpha, \beta}+\sum_{\beta=r_{j+1}+1}^{r_j} \epsilon^{\sigma}_{\beta, \alpha}.
    \end{align*}
    Notice that if $\alpha\neq \beta$, we have $\epsilon^{\sigma}_{\alpha,\beta}+\epsilon^{\sigma}_{\beta,\alpha}=1$. Therefore
    \begin{align*}
 \sum_{k=i}^{j-1}\sum_{\beta=r_{k+1}+1}^{r_{k}}(\epsilon^{\sigma}_{\alpha,\beta}+\epsilon^{\sigma}_{\beta,\alpha})= \sum_{k=i}^{j-1}\sum_{\beta=r_{k+1}+1}^{r_{k}}1=r_i-r_j.
    \end{align*}
    Secondly, we have that
    \begin{align*}
        \sum_{\beta=r_{j+1}+1}^{r_j} \epsilon^{\sigma}_{\beta, \alpha}=\alpha-r_{j+1}-1.
    \end{align*}
    We conclude that
\[
\sum_{k=i}^j h^\sigma_1(k, \alpha)=h^\sigma_0(i,\alpha)+r_i-r_j+\alpha-r_{j+1}-1.\qedhere
\]
\end{proof}

\subsection{Genus 0}
In the case where the curve $C\cong \BP^1$ has genus 0 and $E=\OO_{\BP^1}^{\oplus r}$, we prove a different closed expression for the motive of the hyperquot scheme, where the role of the partial flag variety becomes more evident. We denote by  $\Flag(\ns,r)$   the partial flag variety parametrising flags of vector subspaces
\[
V_l\subset \dots \subset V_1 \subset V,
\]
of a fixed $r$-dimensional vector space $V$, where we set $\dim V_i =r_i=r-s_i$. Setting $r_0=r$ and $r_{l+1}=0$ as ever, an easy inductive argument yields the relation
\begin{align*}
    [\Flag(\ns,r)] = \frac{\prod_{k=1}^{r}\,\left(\BL^k-1\right)}{\prod_{j=0}^{l}\prod_{k=1}^{r_j-r_{j+1}}\,\left(\BL^k-1\right)}\in K_0(\Var_\bfk).
\end{align*}
The inductive argument uses the base case
\[
[G(d,r)]=\frac{\prod_{k=1}^{r}\,\left(\BL^k-1\right)}{\prod_{k=1}^{d}\,\left(\BL^k-1\right)\prod_{k=1}^{r-d}\,\left(\BL^k-1\right)}
\]
and the fibration
\[
\begin{tikzcd}
\Flag(r,s_1,\ldots,s_{l-1},s_l) \arrow{r} & \Flag(r,s_1,\ldots,s_{l-1})
\end{tikzcd}
\]
forgetting $V_l$, which has fibre the Grassmannian $G(r_l,r_{l-1})$.
\begin{theorem}\label{thm: fact genus 0}
There is an identity
\[
\mathsf{Z}_{\BP^1,\OO^{\oplus r}, \ns}(\nq)=  [\Flag(\ns,r)] \prod_{1\leqslant i\leqslant j \leqslant l}\prod_{\alpha=r_{j+1}+1}^{r_j}\frac{1}{\left(1-\BL^{r_i-\alpha} q_i\cdots q_j \right)\left(1-\BL^{r_{i-1}-\alpha+1} q_i\cdots q_j \right)}.
\]
\end{theorem}
\begin{proof}
In the case $\BP^1$ and $E=\OO^{\oplus r}$ the hyperquot scheme is smooth and unobstructed by \Cref{prop: smooth nested flag}. Therefore 
by Theorem  \ref{thm: theorem motivic all g} we have
\begin{align*}
       \mathsf{Z}_{\BP^1,\OO^{\oplus r}, \ns}(\nq)=\sum_{\sigma\in P_{ \nr}}\BL^{\sum_{j=1}^{l}h^\sigma_{2}(j,0, 0,\dots,  0)}\prod_{1\leqslant i\leqslant j\leqslant l}\prod_{\alpha=r_{j+1}+1}^{r_j}\zeta_C\left(\BL^{ h^\sigma_0(i,\alpha)+r_i-r_j+\alpha-r_{j+1}-1}q_i\cdots q_j \right).
\end{align*}
Recall that the motivic zeta function of $\BP^1$ is (see e.g.~\cite{Kapranov_rational_zeta})
\begin{align*}
\zeta_{\BP^1}(q)=\frac{1}{(1-q)(1-\BL q)}.
\end{align*}
Thus we conclude that
\begin{multline*}
\mathsf{Z}_{\BP^1,\OO^{\oplus r}, \ns}(\nq)=\sum_{\sigma\in P_{ \nr}}\BL^{\sum_{j=1}^{l}\sum_{\alpha=r_{j+1}+1}^{r_{j}}\sum_{\beta=r_j+1}^{r}  \epsilon^{\sigma}_{\alpha,\beta}}\\
\cdot \prod_{1\leqslant i\leqslant j\leqslant l}\prod_{\alpha=r_{j+1}+1}^{r_j}\frac{1}{\left(1-\BL^{ h^\sigma_0(i,\alpha)+r_i-r_j+\alpha-r_{j+1}-1}q_i\cdots q_j \right)\left(1-\BL^{ h^\sigma_0(i,\alpha)+r_i-r_j+\alpha-r_{j+1}}q_i\cdots q_j \right)}\\
      = \frac{\prod_{k=1}^{r}\,\left(1-\BL^k\right)}{\prod_{j=0}^{l}\prod_{k=1}^{r_j-r_{j+1}}\,\left(1-\BL^k\right)}\prod_{1\leqslant i\leqslant j \leqslant l}\prod_{\alpha=r_{j+1}+1}^{r_j}\frac{1}{\left(1-\BL^{r_i-\alpha} q_i\cdots q_j \right)\left(1-\BL^{r_{i-1}-\alpha+1} q_i\cdots q_j \right)},
\end{multline*}
where in the last expression we used the purely combinatorial identity \cite[Prop. 5]{Chen_hyperquot}, suitably adapted to our conventions on the indices.
\end{proof}

\begin{remark}
It would be interesting to understand if the  product formula in Theorem \ref{thm: fact genus 0} could be geometrically explained by the existence of a fibration
\[
\begin{tikzcd}
\Quot_{\BP^1}(\OO^{\oplus r}, \ns, \nd)\arrow{r} &  \Flag(\ns,r).
\end{tikzcd}
\]

\end{remark}
\subsection{Irreducibility}
Assume $\bfk$ has characteristic $0$ in this subsection. Specialising the motivic class from  Theorem \ref{thm: theorem motivic all g} to the Poincaré polynomial and extracting the $0$-th Betti number gives a simple criterion to count the number of connected components of the hyperquot scheme, in the smooth and unobstructed case. We do this explicitly in the genus 0 case.
\begin{corollary}\label{cor:connectedness}
The hyperquot scheme $\Quot_{\BP^1}(\OO^{\oplus r}, \ns, \nd)$ is irreducible for all $\ns, \nd$.
\end{corollary}
\begin{proof}
Without loss of generality, we can suppose that $r_{i}> r_j $ for all $0\leqslant  i< j\leqslant l+1$. Denote by $b_0$ the $0$-th Betti number. Its generating series is obtained by the specialisation  $\BL=0$
\begin{align*}
    \sum_{\nd}\,b_0\left(\Quot_{\BP^1}(\OO^{\oplus r}, \ns, \nd)\right)\boldit{q}^{\nd}&=\mathsf{Z}_{\BP^1,\OO^{\oplus r}, \ns}(\nq)|_{\BL=0}\\
    &= \prod_{i=1}^l\frac{1}{1- q_i }=\sum_{\nd}\nq^{\nd},
\end{align*}
  which implies that $ b_0(\Quot_{\BP^1}(\OO^{\oplus r}, \ns, \nd))=1$ and therefore $\Quot_{\BP^1}(\OO^{\oplus r}, \ns, \nd)$ has only one connected component.
\end{proof}
\subsection{Ordinary Quot scheme}
As a special case, we obtain new formulas for the generating series of motives of the classical Quot scheme.
\begin{corollary}
    Let $C$ be a smooth projective curve of genus $g$ and $E=\bigoplus_{\alpha=1}^rL_\alpha$ such that the Quot scheme $ \Quot_C(E, s, d)$ is smooth and unobstructed. Set $\Tilde{r}=r-s$. There is an identity
\begin{align*}
\mathsf{Z}_{C,E, s}(q)=\sum_{\sigma\in P_{ \Tilde{r}}}\BL^{h^\sigma_{2}(1,g, \deg L_1,\dots,  \deg L_r)} q^{ \sum_{\alpha=\Tilde{r}+1}^{r}\deg L_{\sigma(\alpha)}}
\cdot\prod_{\alpha=1}^{\Tilde{r}}\zeta_C\left(\BL^{ h^\sigma_0(i,\alpha)+\alpha-1}q \right).
\end{align*}
In particular, $\mathsf{Z}_{C,E, \ns}(\nq)$ is a rational function. Moreover there is an identity
\begin{align*}
     \mathsf{Z}_{\BP^1,\OO^{\oplus r}, s}(q)=  [G(s,r)] \cdot\prod_{\alpha=1}^{\Tilde{r}}\frac{1}{\left(1-\BL^{\Tilde{r}-\alpha} q\right)\left(1-\BL^{r-\alpha+1} q \right)}.
\end{align*}
\end{corollary}

\appendix
\section{Functoriality of virtual classes}\label{app:functoriality-VFC}
Fix a smooth projective curve $C$, a locally free sheaf $E$, two $l$-tuples $\ns$ and $\nd$, and form the Quot scheme $\QQ = \Quot_C(E,\ns,\nd)$.

Given a locally free sheaf $F$ on $C$, define for each $i=1,\ldots,l$ the associated \emph{tautological complex} on $\QQ$ by 
\begin{align*}
F^{[\nd,i]}=\RR\pi_*\left(p^*F\otimes \CT_i\right),
\end{align*}
where $\CT_i$ is the $i$-th universal quotient over $ \QQ\times C$ and $\pi, p$ are the natural projections to $\QQ$ and $C$ respectively. By cohomology and base change, the fibre over a point $z=[E\onto T_l\onto \cdots \onto T_1]\in\QQ$ is the complex
\[
F^{[\nd,i]}\big|_{z}=\RR\Gamma(C, F\otimes T_i),
\]
which implies that $F^{[\nd,i]} \in \Perf^{[0,1]}(\QQ)$.
Note that
\begin{equation}\label{eqn:basechange-cohomology}
\HH^1(C, F\otimes T_i) = h^1(F^{[\nd,i]}\big|_{z}) = h^1(F^{[\nd,i]})\big|_z
\end{equation}
since there is no higher cohomology.

Now assume that $E$ sits in a short exact sequence
\[
0 \to E_1 \to E \to E_2 \to 0
\]
of locally free sheaves. Set $\QQ_2 = \Quot_C(E_2,\ns,\nd)$. Precomposing with $E \onto E_2$ yields a natural closed immersion
\[
\begin{tikzcd}
\QQ_2 \arrow[hook]{r}{\iota} & \QQ,
\end{tikzcd}
\]
which realises $\QQ_2$ as the zero locus of a natural section 
\[
\sigma \in \HH^0(\QQ,h^0(F_1^{[\nd,l]})), \qquad F_1 = E_1^\ast.
\]
The proof of this assertion goes along the same lines of the one of \cite[Thm.~2.8]{Mon_double_nested}.
We remark that a similar situation was studied in \cite[Sec.~4.1]{stark2024quot} for the Quot scheme of $0$-dimensional quotients on a smooth projective surface.

Our next goal is to construct a morphism of exact triangles (which we will call a `compatibility diagram')
\begin{equation}\label{diag:compatible-triple}
\begin{tikzcd}
\iota^\ast \BF_{\QQ}\arrow{r}{\psi}\arrow{d} & \BF_{\QQ_2}\arrow{r}\arrow{d} & \BF_{\QQ_2/\QQ}\arrow{d} \\
\iota^\ast \BL_{\QQ}\arrow{r} & \BL_{\QQ_2}\arrow{r} & \BL_{\QQ_2/\QQ}
\end{tikzcd}
\end{equation}
where $\BF_{\QQ}$ and $\BF_{\QQ_2}$ are the complexes \eqref{def:Gamma-cone}, and the rightmost square is induced by the leftmost one. Once we have this, and assuming $\BF_{\QQ} = \BE_{\QQ}$ (cf.~\Cref{rmk:the-two-cones}), by \Cref{thm: pot 2 nest} the map 
\[
\begin{tikzcd}
\BF_{\QQ_2/\QQ}\arrow{r} & \BL_{\QQ_2/\QQ}
\end{tikzcd}
\]
is automatically a relative obstruction theory for $\iota\colon \QQ_2 \into \QQ$ by \cite[Constr.~3.13]{Manolache-virtual-pb}. Under suitable additional assumptions, we can, and will, make sure $\BF_{\QQ_2/\QQ}$ is concentrated in degree $-1$, equal to the shifted locally free sheaf of the form $F^\ast\big|_{\QQ_2}[1]$, where $F= h^0(F_1^{[\nd,l]}) = \pi_\ast(p^\ast F_1 \otimes \CT_l)$. Then, in this situation, \cite{KKP} will give us the compatibility relation
\[
\iota_\ast[\QQ_2]^{\vir} = e(F)\cap [\QQ]^{\vir}\, \in \, A_\ast(\QQ)
\]
between virtual classes. See \Cref{prop:PF-vir} for the precise statement. 

\smallbreak
Let us go back to the contruction of $\psi$. Let us denote by $\CK_i$ and $\CT_i$ the universal kernels and the universal quotients, respectively, on $\QQ \times C$. Similarly, $\CK_i^2$ and $\CT_i^2$ are the universal kernels and universal quotients on $\QQ_2 \times C$. Consider, to fix the notation, the diagram
\[
\begin{tikzcd}
& \QQ_2 \times C\arrow{dr}{p_2}\arrow[hook]{rr}{\iota_C}\arrow[swap]{dl}{\pi_2} & & \QQ \times C\arrow[swap]{dl}{p}\arrow{dr}{\pi} \\
\QQ_2 & & C & & \QQ & 
\end{tikzcd}
\]
and note that one has $\iota_C^\ast E_{\QQ} = p_2^\ast E$. Consider, for each $i=1,\ldots,l$, the morphism of exact sequences
\[
\begin{tikzcd}
0\arrow{r} & \iota_C^\ast\CK_i\arrow{r}\arrow{d} & \iota_C^\ast E_{\QQ}\arrow{r}\arrow[two heads]{d} & \iota_C^\ast \CT_i\isoarrow{d}\arrow{r}\arrow{d} & 0 \\
0\arrow{r} & \CK_i^2\arrow{r} & p_2^\ast E_2\arrow{r} & \CT_i^2\arrow{r} & 0
\end{tikzcd}
\]
in $\Coh(\QQ_2\times C)$, where, to confirm exactness on the left for the top sequence, we use the vanishings of the higher left derived restriction $\mathbf{L}_{>0}\iota_C^\ast$ on all $\QQ$-flat sheaves such as $\CT_i$.

The kernel of the middle vertical surjection is $p_2^\ast E_1$, therefore we obtain an exact triangle
\[
\begin{tikzcd}
p_2^\ast E_1 \arrow{r} & \iota_C^\ast\CK_i\arrow{r} & \CK_i^2 \arrow{r} & p_2^\ast E_1[1]
\end{tikzcd}
\]
in the derived category of $\QQ_2 \times C$. Applying $\RRlHom_{\pi_2}(-,\CT_i^2)$ yields an exact triangle
\[
\begin{tikzcd}
\RRlHom_{\pi_2}(\CK_i^2,\CT_i^2) \arrow{r}{a_i} &
\RRlHom_{\pi_2}(\iota_C^\ast\CK_i,\CT_i^2) \arrow{r} &
\RRlHom_{\pi_2}(p_2^\ast E_1,\CT_i^2)
\end{tikzcd}
\]
in the derived category of $\QQ_2$. The middle term is 
\[
\RRlHom_{\pi_2}(\iota_C^\ast\CK_i,\iota_C^\ast \CT_i) = \RR\pi_{2\ast} \iota_C^\ast \RRlHom(\CK_i,\CT_i) = \iota^\ast \RRlHom_\pi(\CK_i,\CT_i).
\]
The third term is, after setting $F_1 = E_1^\ast$,
\[
\RR \pi_{2\ast}\RRlHom(p_2^\ast E_1,\CT_i^2) = \RR \pi_{2\ast} (p_2^\ast F_1\otimes \CT_i^2) = F_1^{[\nd,i]} \in \Perf^{[0,1]}(\QQ_2).
\]
We have thus obtained natural exact triangles 
\[
\begin{tikzcd}
\RRlHom_{\pi_2}(\CK_i^2,\CT_i^2)\arrow{r}{a_i} &   \iota^\ast \RRlHom_\pi(\CK_i,\CT_i)\arrow{r} & F_1^{[\nd,i]},\quad i=1,\ldots,l,
\end{tikzcd}
\]
in the derived category of $\QQ_2$. The same construction, namely applying $\RRlHom_{\pi_2}(-,\CT_i^2)$ to the exact triangle
\[
\begin{tikzcd}
p_2^\ast E_1 \arrow{r} & \iota_C^\ast\CK_{i+1}\arrow{r} & \CK_{i+1}^2 \arrow{r} & p_2^\ast E_1[1]
\end{tikzcd}
\]
gives rise to natural exact triangles
\[
\begin{tikzcd}
\RRlHom_{\pi_2}(\CK_{i+1}^2,\CT_i^2)\arrow{r}{b_i} & \iota^\ast \RRlHom_\pi(\CK_{i+1},\CT_i) \arrow{r} & F_1^{[\nd,i]}, \quad i=1,\ldots,l-1,
\end{tikzcd}
\]
in the derived category of $\QQ_2$. Taking their direct sums we obtain a diagram of (horizontal and vertical) exact triangles
\[
\begin{tikzcd}[column sep=large,row sep=large]
\displaystyle\bigoplus_{i=1}^l\RRlHom_{\pi_2}(\CK_i^2,\CT_i^2)\arrow{d}{\RR_{\pi_2}\Delta}\arrow{r}{\mathrm{diag}(a_i)} & \displaystyle\bigoplus_{i=1}^l\iota^\ast \RRlHom_\pi(\CK_i,\CT_i)\arrow{d}{\iota^\ast \RR_\pi\Delta}\arrow{r} & \displaystyle\bigoplus_{i=1}^{l} F_1^{[\nd,i]}\arrow{d} \\
\displaystyle\bigoplus_{i=1}^{l-1}\RRlHom_{\pi_2}(\CK_{i+1}^2,\CT_i^2)\arrow{d}\arrow{r}{\mathrm{diag}(b_i)} & \displaystyle\bigoplus_{i=1}^{l-1}\iota^\ast \RRlHom_\pi(\CK_{i+1},\CT_i)\arrow{d} \arrow{r} & \displaystyle\bigoplus_{i=1}^{l-1} F_1^{[\nd,i]}\arrow{d} \\
\Gamma_2\arrow{r} & \iota^\ast \Gamma \arrow{r} & \Gamma_{\QQ_2/\QQ}
\end{tikzcd}
\]
andapplying $(-)^\vee[1]$ to the bottom horizontal triangle gives the triangle
\[
\begin{tikzcd}
\Gamma_{\QQ_2/\QQ}^\vee[1]\arrow{r} & (\iota^\ast\Gamma)^\vee[1] = \iota^\ast(\Gamma^\vee[1]) = \iota^\ast \BF_{\QQ} \arrow{r}{\psi} & \Gamma_2^\vee[1] = \BF_{\QQ_2} \arrow{r} & \Gamma_{\QQ_2/\QQ}^\vee[2].
\end{tikzcd}
\]
Note that
\[
\Gamma_{\QQ_2/\QQ}^\vee[1] = (F_1^{[\nd,l]})^\vee\,\in\,\Perf^{[-1,0]}(\QQ_2),
\]
the dual of the $l$-th tautological complex associated to $F_1=E_1^\ast$ (calculated on $\QQ_2$).
Now set
\[
\BF_{\QQ_2/\QQ} = \Cone(\psi) = \Gamma_{\QQ_2/\QQ}^\vee[2] = (F_1^{[\nd,l]})^\vee[1]\,\in\,\Perf^{[-2,-1]}(\QQ_2).
\]
We have obtained the sought after compatible triple \eqref{diag:compatible-triple}, and in particular a morphism of complexes
\begin{equation}\label{eqn:relative-POT}
\begin{tikzcd}
\BF_{\QQ_2/\QQ}\arrow{r} & \BL_{\QQ_2/\QQ}.
\end{tikzcd}
\end{equation}

\begin{prop}\label{prop:PF-vir}
Assume that $\BE_{\QQ}=\BF_{\QQ}$ and $\BE_{\QQ_2}=\BF_{\QQ_2}$ (cf.~\Cref{rmk:the-two-cones}), so that \eqref{eqn:relative-POT} is a relative perfect obstruction theory for $\iota\colon \QQ_2\into \QQ$.
Assume, in addition, that for all closed points $z=[E\onto T_l\onto \cdots \onto T_1]\in \QQ$ we have
\[
\HH^0(C,E_1\otimes T_l^{\mathrm{fr},\ast}\otimes \omega_C)=0,
\]
where $T_l^{\mathrm{fr}}$ denotes the free part of $T_l$. Then $F_1^{[\nd,l]}\cong h^0(F_1^{[\nd,l]}) \in \Coh(\QQ)$ is a locally free sheaf of rank $(1-g)s_l\rk E_1-s_l\deg E_1+d_l\rk E_1$. Denote this sheaf by $F$. Then we also have identities
    \begin{align*}
\iota_*\OO^{\vir}_{\QQ_2}&=\Lambda^{\bullet}\left(F^\ast\right)\otimes \OO^{\vir}_{\QQ}\in K_0(\QQ),\\
\iota_*[\QQ_2]^{\vir}&=e\left(F\right)\cap [\QQ]^{\vir}\in A_*(\QQ)
    \end{align*}
in $K$-theory and Chow, respectively.
\end{prop}

\begin{proof}
Our assumption, combined with Serre duality and \Cref{eqn:basechange-cohomology}, says that 
\[
h^1(F_1^{[\nd,l]})\big|_z=0,
\]
from which we deduce that 
\[
F_1^{[\nd,l]} = h^0(F_1^{[\nd,l]}) = \pi_\ast(p^\ast F_1 \otimes \CT_l)\,\in\,\Coh(\QQ)
\]
is locally free. Let us set 
\[
F = F_1^{[\nd,l]}, \qquad \V(F) = \Spec_{\OO_{\QQ}} \Sym F^\ast.
\]
We know that a canonical section $\sigma \in \HH^0(\QQ,\V(F)) = \Hom_{\OO_{\QQ}}(\OO_{\QQ},F)$ cuts out $\QQ_2$. In other words, we have a cartesian diagram
\[
\begin{tikzcd}[row sep=large,column sep=large]
\QQ_2\MySymb{dr} \arrow[hook]{r}{\iota}\arrow[hook,swap]{d}{\iota} & \QQ\arrow[hook]{d}{\sigma} \\
\QQ \arrow[hook,swap]{r}{0} & \V(F)
\end{tikzcd}
\]
where `$0$' denotes the zero section. We claim that 
\begin{equation}\label{0^!-formula}
0^![\QQ]^{\vir} = [\QQ_2]^{\vir}
\end{equation}
in $A_\ast(\QQ_2)$.
Note that, since the zero section is a regular immersion, its cotangent complex is 
\[
\BL_{\QQ/\V(F)} = (\mathscr I/\mathscr I^2)[1] = F^\ast[1],
\]
where $\mathscr I$ is the ideal sheaf of the zero section.
Its restriction to $\QQ_2$ is
\[
\iota^\ast \BL_{\QQ/\V(F)} = F^\ast\big|_{\QQ_2}[1] = \BF_{\QQ_2/\QQ}.
\]
We are then in the situation of a compatible diagram
\[
\begin{tikzcd}
\iota^\ast \BF_{\QQ}\arrow{r}{\psi}\arrow{d} & \BF_{\QQ_2}\arrow{r}\arrow{d} & \iota^\ast \BL_{\QQ/\V(F)}\arrow{d} \\
\iota^\ast \BL_{\QQ}\arrow{r} & \BL_{\QQ_2}\arrow{r} & \BL_{\QQ_2/\QQ}
\end{tikzcd}
\]
and \cite[Thm.~1]{KKP} gives the identity \eqref{0^!-formula}. Pushing \eqref{0^!-formula} forward in Chow and in $K$-theory gives the required identities.
\end{proof}
\begin{remark}
We expect to be possible to drop the  assumption that $F_1^{[\nd,l]}$ is a locally free sheaf, but only a perfect 2-term complex, by using a recent result of Oh--Thomas \cite{OT_quantum}, in order to obtain analogous pushforward formulae.
\end{remark}

\section{\texorpdfstring{Virtual $\chi_{-y}$-genus}{}}\label{app: chi_y}
Consider $\bfk=\BC$ for this section.
Let $C$ be a smooth projective curve, $E$ a locally free sheaf of rank $r$ and fix tuples $\ns=(s_1\leqslant \dots \leqslant s_l)$ and $\nd=(d_1, \dots, d_l)$. Notice that by deformation invariance, we can always assume that there exists a splitting $E=\bigoplus_{\alpha=1}^r L_\alpha$. We proved in \Cref{cor: virtual cycles} that the hyperquot scheme $\QQ=\Quot_C(E, \ns, \nd)$ admits a virtual structure sheaf
\begin{align*}
    \OO_{\QQ}^{\vir}\in K_0(\QQ).
\end{align*}
The \emph{virtual tangent bundle} is the class in $K$-theory of the dual of the perfect obstruction theory of \Cref{thm: pot 2 nest}, which by \Cref{rmk:same-K} can be computed as
\begin{align*}
    T_{\QQ}^{\vir}=\sum_{i=1}^l \RRlHom_\pi (\CK_i,\CT_i)-\sum_{i=1}^{l-1} \RRlHom_\pi (\CK_{i+1},\CT_i) \in K^0(\QQ).
\end{align*}
For any locally free sheaf $V$ on a projective scheme $X$, define the class in $K$-theory
\[
\Lambda_{t}V=\sum_{i=0}^{\rk V}t^i\cdot \Lambda^i V\in K^0(X)[t],
\]
and extend it by linearity, for any $V\in K^0(X)$, to a class $\Lambda_{t}V\in K^0(X)[\![t]\!]$ (see e.g. \cite[Def. 5.1]{FG_riemann_roch}).
Following \cite{FG_riemann_roch}, we define the \emph{virtual $\chi_{-y}$-genus} of $\Quot_C(E, \ns, \nd)$ as\footnote{The $\chi_{-y}$-genus is polynomial in $y$ by \cite[Prop. 5.3]{FG_riemann_roch}.}
\begin{align*}
    \chi^{\vir}_{-y}(\Quot_C(E, \ns, \nd))=\chi\left(\Quot_C(E, \ns, \nd), \OO^{\vir}\otimes \Lambda_{-y} (T^{\vir})^\vee \right)\in \BZ[y].
\end{align*}
We define the generating series of $\chi_{-y}$-genera
\begin{align*}
    Z_{C, E, \ns}(\nq,y)=\sum_{\nd}\nq^{\nd}\cdot\chi^{\vir}_{-y}(\Quot_C(E, \ns, \nd))\in \BZ[y](\!(\nq)\!),
\end{align*}
where we used the multinotation $\nq^{\nd}=q_1^{d_1}\cdots q_l^{d_l}$. This generating series is well defined as a Laurent power series in $\nq$, since the hyperquot scheme $ \Quot_C(E, \ns, \nd)$ is empty for $\nd$ sufficiently negative.

As an application, we prove that the $\chi_{-y}$-genus of the hyperquot schemes only depends on the genus of the curve and the degrees of the line bundles $L_\alpha$. We remark that, due to the presence of nesting conditions, we cannot apply verbatim the standard techniques of \cite{EGL_cobordism}.
\begin{prop}\label{prop: poly} The virtual $\chi_{-y}$-genus $\chi_{-y}^{\vir}(\Quot_C(E, \ns, \nd)) $ depends on $(C, L_1, \dots, L_r)$ only  on the genus $g=g(C)$ and the degrees $\deg L_\alpha$, for $\alpha=1, \dots, r$.
\end{prop}
\begin{proof}
    Consider the action of the torus $\TT$ on $\Quot_C(E, \ns, \nd)$ as in Section \ref{sec: torus action}. To simplify the notation, we write 
    \[
     \Quot_C(E, \ns, \nd)^\TT=\coprod_{\underline{\ns}, \underline{\nd}} Q_{\underline{\ns}, \underline{\nd}}
    \]
    the decomposition of the fixed locus of Lemma \ref{lemma: fixed locus}, where 
\[
Q_{\underline{\ns}, \underline{\nd}}=\prod_{\alpha=1}^r   \Quot_C(L_\alpha,\ns_\alpha, \nd_\alpha).
\]
Denote by $w_1,\ldots,w_r$ the irreducible representations of the torus $\TT$.

 By the $K$-theoretic virtual localisation \cite{FG_riemann_roch}, each fixed component of $ \Quot_C(E, \ns, \nd)$ inherits a virtual structure sheaf and 
    \begin{align}\label{eqn: localized contrub chi y}
         \chi^{\vir}_{-y}(\Quot_C(E, \ns, \nd))=\left.\left(\sum_{\underline{\ns}, \underline{\nd}}\chi\left(  Q_{\underline{\ns}, \underline{\nd}}, \OO^{\vir}\otimes \frac{\Lambda_{-y} (T^{\vir}|_{ Q_{\underline{\ns}, \underline{\nd}}})^\vee }{\Lambda_{-1}(N^{\vir}_{Q_{\underline{\ns}, \underline{\nd}}})^{\vee} }\right)\right)\right|_{w_\alpha=1}.
    \end{align}
   Here, for each component of the fixed locus $\chi(\cdot)$ is performed $\TT$-equivariantly, and the (restriction of the) virtual tangent bundle is
   \[
   T^{\vir}|_{ Q_{\underline{\ns}, \underline{\nd}}}=T^{\vir}_{ Q_{\underline{\ns}, \underline{\nd}}}+N^{\vir}_{Q_{\underline{\ns}, \underline{\nd}}}\in K^0(Q_{\underline{\ns}, \underline{\nd}}), 
   \]
   where the \emph{virtual normal bundle} $ N^{\vir}_{Q_{\underline{\ns}, \underline{\nd}}}$ is the $\TT$-movable part and the virtual tangent bundle $T^{\vir}_{ Q_{\underline{\ns}, \underline{\nd}}}$ of $ Q_{\underline{\ns}, \underline{\nd}}$ is the $\TT$-fixed part by \cite{GP_virtual_localization}.

   Fix now $\underline{\ns}, \underline{\nd}$ and consider the cartesian diagram
   \begin{center}
        \begin{tikzcd}
       Q_{\underline{\ns}, \underline{\nd}}\times C\arrow[r, "\iota"]\arrow[d, "\pi"]&  \Quot_C(E, \ns, \nd) \times C\arrow[d, "\pi"]\\
         Q_{\underline{\ns}, \underline{\nd}}\arrow[r, "\iota"]&  \Quot_C(E, \ns, \nd). 
   \end{tikzcd}
   \end{center}

   The universal sheaves $\CK_i, \CT_i$ on $ \Quot_C(E, \ns, \nd)\times C$ restrict via $\iota$ as
   \begin{align*}
       \CK_i&=\bigoplus_{\alpha=1}^r \CK_{i, \alpha}\\
       \CT_i&=\bigoplus_{\alpha=1}^r \CT_{i, \alpha},
   \end{align*}
   where $\CK_{i, \alpha}, \CT_{i, \alpha}$ are the universal sheaves on each factor $ \Quot_C(L_\alpha,\ns_\alpha, \nd_\alpha)\times C$ of $Q_{\underline{\ns}, \underline{\nd}}\times C$.

   The restriction of the virtual tangent bundle equivariantly decomposes as
 \begin{align*}
    T^{\vir}|_{Q_{\underline{\ns}, \underline{\nd}}}=\sum_{\alpha, \beta=1}^r\left(\sum_{i=1}^l \RRlHom_\pi (\CK_{i, \alpha},\CT_{i, \beta})-\sum_{i=1}^{l-1} \RRlHom_\pi (\CK_{i+1, \alpha},\CT_{i, \beta})\right) w_\beta w_\alpha^{-1} \in K^0(Q_{\underline{\ns}, \underline{\nd}}),
\end{align*}
and its $\TT$-fixed part is
\begin{align*}
     T^{\vir}_{Q_{\underline{\ns}, \underline{\nd}}}=\sum_{\alpha=1}^r\left(\sum_{i=1}^l \RRlHom_\pi (\CK_{i, \alpha},\CT_{i, \alpha})-\sum_{i=1}^{l-1} \RRlHom_\pi (\CK_{i+1, \alpha},\CT_{i, \alpha})\right).
\end{align*}
By the discussions of Section \ref{sec: comb fixed} it is easy to see that, whenever 
$Q_{\underline{\ns}, \underline{\nd}} $ is non-empty, it is isomorphic to a product
\begin{align}\label{eqn: Qn as prod of hilb nested}
    Q_{\underline{\ns},  \underline{\nd}}\cong \prod_{\alpha}C^{[\nn_\alpha]},
\end{align}
where each $\nn_\alpha=(n_{1, \alpha}\leqslant \dots \leqslant n_{k_{\alpha}, \alpha} )$ for some $k_{\alpha}\geqslant 1$. Therefore $Q_{\underline{\ns},  \underline{\nd}}$ is smooth, unobstructed and the virtual tangent bundle coincides with the  actual tangent bundle (as classes in $K$-theory)
\[
 T^{\vir}_{Q_{\underline{\ns},  \underline{\nd}}}= T^{}_{Q_{\underline{\ns},  \underline{\nd}}}\in K^0(Q_{\underline{\ns}, \underline{\nd}}),
\]
which implies that the virtual structure sheaf coincides with the structure sheaf by \cite[Cor. 4.5]{Tho_K-theo_Fulton}. On the product $ \prod_{\alpha}C^{[\nn_\alpha]}\times C$, there are universal exact sequences of sheaves
\[
\begin{tikzcd}
0\arrow{r} &  \CI_{i,\alpha}\arrow{r} & \OO\arrow{r} & \OO_{\CZ_{i,\alpha}} \arrow{r} & 0 
\end{tikzcd}
\]
where we suppressed some obvious pullback maps. Under the identification \eqref{eqn: Qn as prod of hilb nested}, the universal sequences of sheaves on $Q_{\underline{\ns},  \underline{\nd}}\times C$
\[
\begin{tikzcd}
0\arrow{r} & \CK_{i, \alpha}\arrow{r} & L_\alpha\arrow{r} &\CT_{i, \alpha}\arrow{r} & 0
\end{tikzcd}
\]
are of one of the following two forms
\begin{equation}\label{eqn: one of two options}
\begin{tikzcd}[row sep=tiny]
0 \arrow{r} & L_\alpha\otimes \CI_{i,\alpha} \arrow{r} & L_\alpha \arrow{r} &L_\alpha\otimes \OO_{\CZ_{i,\alpha}}\arrow{r} & 0 \\
0 \arrow{r} & L_\alpha\arrow{r}{\sim} & L_\alpha\arrow{r} & 0.
\end{tikzcd}
\end{equation}
Using \eqref{eqn: one of two options}, we have that under the identification \eqref{eqn: Qn as prod of hilb nested}, for each $\alpha, \beta=1,\dots, r$ and $ i, j=1, \dots, l$ we can represent the classes 
\[
 \RRlHom_\pi (\CK_{i, \alpha},\CT_{j, \beta})\in K^0\left(Q_{\underline{\ns},  \underline{\nd}}\times C\right)
\]
as a linear combination of classes
\begin{align*}
 \RRlHom_\pi(\CI_{i,\alpha}, \CI_{j,\beta}\otimes L_\alpha^{-1}L_\beta ),\, \RRlHom_\pi(\CI_{i,\alpha}, \OO\otimes L_\alpha^{-1}L_\beta )  \in  K^0\left(\prod_{\alpha}C^{[\nn_\alpha]}\times C\right).
\end{align*}
Consider now a tuple $\nn_\alpha=(n_{1, \alpha}\leqslant \dots \leqslant n_{k_{\alpha}, \alpha} )$. We have the isomorphism
\begin{align}\label{eqn: iso nested with product in proof}
    C^{[\nn_\alpha]}\cong C^{(n_{1, \alpha})}\times C^{(n_{2, \alpha}-n_{1, \alpha})}\times \dots \times C^{(n_{k_{\alpha}, \alpha}-n_{k_{\alpha}-1,\alpha})}.
\end{align}
Again, on each factor $C^{(n_{i, \alpha}-n_{i-1, \alpha})}\times C$, there is a universal sequence
\[
\begin{tikzcd}
0 \arrow{r} & \CJ_{i,\alpha}\arrow{r} & \OO \arrow{r} & \OO_{\CW_{i,\alpha}}\arrow{r} & 0,
\end{tikzcd}
\]
and under the isomorphism \eqref{eqn: iso nested with product in proof}, we have the identification
\[
\CI_{i, \alpha}\cong \CJ_{i,\alpha}\otimes \CJ_{i-1,\alpha}\otimes \dots \otimes \CJ_{1,\alpha}.
\]
To sum up, we showed that 
\[
Q_{\underline{\ns}, \underline{\nd}}\cong \prod_{k\in I}C^{(n_k)},
\]
for a suitable family of indices $I$ and suitable $(n_k)_{k\in I}$, and that under this identification, both $T^{\vir}|_{ Q_{\underline{\ns}, \underline{\nd}}}, N^{\vir}_{Q_{\underline{\ns}, \underline{\nd}}}$ are linear combination of classes of the form
\begin{align}\label{eqn: chern classes of the form uni}
    \RRlHom_\pi\left(\bigotimes_{k\in K}\CJ_{k}, \bigotimes_{m\in M}\CJ_{m}\otimes L_{a}\right),
\end{align}
for some families of indices $K, M$ and where $L_a$ is a line bundle on $C$ obtained as a linear combination of the $\set{L_\alpha}_{\alpha}$.

We are now ready to show that each localized contribution of \eqref{eqn: localized contrub chi y} can be expressed as a polynomial on $g=g(C)$ and the degrees $\deg L_\alpha$. By the fact that $\OO^{\vir}=\OO$ and ($\TT$-equivariant) Grothendieck--Riemann--Roch, we have
\begin{align*}
    \chi\left.\left(  Q_{\underline{\ns}, \underline{\nd}}, \OO^{\vir}\otimes \frac{\Lambda_{-y} (T^{\vir}|_{ Q_{\underline{\ns}, \underline{\nd}}})^\vee }{\Lambda_{-1}(N^{\vir}_{Q_{\underline{\ns}, \underline{\nd}}})^{\vee} }\right)\right)&=\int_{ Q_{\underline{\ns}, \underline{\nd}}}\frac{\mathrm{ch}\left( \Lambda_{-y} (T^{\vir}|_{ Q_{\underline{\ns}, \underline{\nd}}})^\vee \right)}{\mathrm{ch}\left( \Lambda_{-1} (N^{\vir}_{Q_{\underline{\ns}, \underline{\nd}}})^\vee \right)}\cdot \mathrm{td}(T_{Q_{\underline{\ns}, \underline{\nd}}})\\
    &=\int_{ \prod_{k\in I}C^{(n_k)}} f,
\end{align*}
where $f$ is a polynomial in the Chern classes of $K$-theory classes of the form \eqref{eqn: chern classes of the form uni}. We conclude by \cite[Prop. 5.3]{Mon_double_nested} that the dependence of the last integral is only on the genus $g=g(C)$ and the degrees of the line bundles $L_\alpha $.
\end{proof}

\bibliographystyle{amsplain-nodash}
\bibliography{The_Bible}

\bigskip
\noindent
{\small Sergej Monavari \\
\address{\'Ecole Polytechnique F\'ed\'erale de Lausanne (EPFL),  CH-1015 Lausanne (Switzerland)} \\
\href{mailto:sergej.monavari@epfl.ch}{\texttt{sergej.monavari@epfl.ch}}
}

\bigskip

\noindent
{\small Andrea T. Ricolfi \\
\address{SISSA, Via Bonomea 265, 34136, Trieste (Italy)} \\
\href{mailto:aricolfi@sissa.it}{\texttt{aricolfi@sissa.it}}
}

\end{document}